\documentclass[11pt]{amsart}
\input{amssym.def}
\input{amssym}
\newtheorem{pro}{Proposition}[section]
\newtheorem{recursion}[pro]{Recursion}
\newtheorem{lem}[pro]{Lemma}

\newtheorem{theo}[pro]{Theorem}
\newtheorem{defi}[pro]{Definition}
\newtheorem{cor}[pro]{Corollary}
\newtheorem{remk}[pro]{Remark}
\newtheorem{assu}[pro]{Assumption}

\newcommand{\ep}{\varepsilon}
\newcommand{\om}{\omega}
\newcommand{\vp}{\varphi}
\newcommand{\la}{\lambda}
\newcommand{\al}{\alpha}

\newcommand{\lra}{\longrightarrow}
\newcommand{\lmt}{\longmapsto}

\newcommand{\nrm}[1]{\mbox{ $ \displaystyle \left\| {#1} \right\| $} }

\newcommand{\nrE}[1]{\mbox{ $ \nrm{ {#1} }_{E} $} }

\newcommand{\fk}[1]{ \left( {#1} \right) }

\newcommand{\bk}[1]{ \left\{ {#1} \right\} }
\newcommand{\btr}[1]{\mbox{ $ \left| {#1} \right| $ }}

\newcommand{\re}{{\mathbf{R}}}
\newcommand{\rep}{{\mathbf{R^+}}}
\newcommand{\za}{{\mathbf{N}}}

\newcommand{\jz}{{\mathbf{J}}}
\newcommand{\ji}{{\mathbf{I}}}

\newcommand{\nres}[1]{e^{-\frac{#1}{\lambda}}}
\newcommand{\nresl}[1]{\frac{e^{-\frac{#1}{\lambda}}}{\la}}

\newcommand{\ilm}[1]{  \lim_{ {#1} \to \infty}  }

\newcommand{\Funk}[5]{ \begin{array}{ccccc}
                       {#1} & : & {#2} & \lra & {#3} \\
                            &   & {#4} & \lmt & \displaystyle{#5} 
                       \end{array}                       }

\newcommand{\Jl}{J_{\lambda}}
\newcommand{\Jlw}{J_{\lambda}^{\om}}

\newcommand{\ul}{u_{\lambda}}
\newcommand{\Tlp}{T_{\la,\vp}}
\newcommand{\up}[2]{u_{#1,#2}}

\newcommand{\hw}{h^{\om}}
\newcommand{\Lw}{L^{\om}}
\newcommand{\dtl}{\fk{\partial_t}_{\la}}

\begin{document}

\title{Existence and Asymptotics of  Abstract Functional Differential Equations}
\author{Josef Kreulich\ Universit\"at Duisburg/Essen\\}
\date{}

\begin{abstract}
It is shown how the linear method of the Yosida-approximation of the derivative applies to solve possibly nonlinear abstract functional differential equations in both, the finite and infinite delay case. A generalization of the integral solution will provide regularity results. Moreover, this method applies to derive uniform convergence on the halfline, and therefore general results on boundedness and various types of asymptotic almost periodicity.
\end{abstract}


\maketitle
\section{Introduction}
To prove the existence and stability of solutions to nonlinear functional differential equation for $r,T>0,$  and  $X$ a Banach space let
\begin{eqnarray*}
\ji &\in & \bk{(-\infty,0], [-r,0]},\\
E&=& BUC(\ji,X),\\
Y&\in& \bk{BUC((-\infty,T],X),C([-r,T],X), BUC(\re,X), BUC([-r,\infty),X)}.
\end{eqnarray*}

For $\jz \in \bk{[0,T],[0,\infty)},$  $\vp\in E,$ and $A(\cdot,\cdot)\subset X\times X$ a family of disspative operators we consider the functional differential equation

\begin{equation} \label{first-FDE}
\left\{
\begin{array}{rcl} 
u^{\prime}(t) &\in & A(t,u_t)u(t)+ \om u(t): \ t \in \jz \\
u_{|\ji}&=&\vp.
\end{array}
\right.
\end{equation}

The first results on this general type of equations were given by Kartsatos and Parrott \cite{KartsatosParrot}, where either the finite or infinite delay case was considered. Additionally, they found a so called generalized solution, for which they show that in the case of reflexive Banach spaces it becomes a strong solution.

In the present study we show how the method of Yosida approximation of the derivative applies to obtain existence of mild and integral solutions to (\ref{first-FDE}) in general Banach spaces, when considering the corresponding non-autonomous Cauchy problem,
\begin{equation} \label{first-Cauchy-Eq}
\left\{
\begin{array}{rcl} 
u^{\prime}(t) &\in & B(t)u(t)+ \om u(t): \ t \in \jz \\
u_{0}&=&\vp(0),
\end{array}
\right.
\end{equation}
where $B(t)=A(t,u_t)$. Therefore the regularity results for this class of solutions become applicable. These results are obtained for the finite and infinite delay case with a single proof.   
Additionally, results on the asymptotic behavior in the finite and infinite delay case are derived with the existence proof. This is a new aspect, since the solutions to a functional differential equation  with an initial state and the corresponding whole line problem are not necessarily asymptotically close. Therefore, in case of infinite delay a different approach to the asymptotics is needed. The main method is an application of the results provided in the study \cite{Kreulichevo}. That is, a reduction to methods coming with Yosida approximations of the derivative for non-autonomous Cauchy problems. This abstraction shows the power of linear analysis and their interaction with control functions given for nonlinear operators in the Assumptions \ref{IVFDE_2_og}, \ref{IVFDE_2}, \ref{general-IVE1}, and \ref{general-IVE2}.

\section{Abstract Functional Differential Equations} \label{Abstract-FDEs}

In the underlying paper the topic is the existence, stability and asymptotic behavior of nonlinear functional differential equations. With the notations as above, we have the canonical embedding:
$$
\Funk{\iota}{BUC(\ji\cup\jz,X)}{BUC(\ji,X)}{f}{f_{|\ji,}}
$$
with respect to the finite or infinite delay. For the definition of the equation let $\om \in \re$ and $\vp\in E.$ We consider  the functional differential equation:
\begin{equation} \label{gFDEonInterval}
\left\{
\begin{array}{rcl} 
u^{\prime}(t) &\in & A(t,u_t)u(t)+ \om u(t): \ t \in \jz \\
u_{|\ji}&=&\vp
\end{array}
\right.
\end{equation}

The main assumptions to solve the problem (\ref{gFDEonInterval}) on the operator $A$ are:
\begin{assu} \label{IVFDE_1}
The family $\bk{A(t,\vp):t \in \jz, \vp \in E}$ are m-dissipative operators.
\end{assu}
\begin{assu}\label{IVFDE_2_og}
There exist bounded and uniformly continuous functions $h,k:\jz \to X,$
 a constant $K_0 >0,$ and $L_1,L_2: \rep \lra \rep$ continuous and monotone non decreasing,
such that for $\la >0$ and $t_1,t_2\in \jz $ we have
\begin{eqnarray*}
\lefteqn{\nrm{x_1-x_2}} \\
&\le& \nrm{x_1-x_2 -\la(y_1-y_2)}+\la \nrm{h(t_1)-h(t_2)}L_1(\nrm{x_2}) \\
&&+ \la\nrm{k(t_1)-k(t_2)}L_2(\nrm{\vp_2})+K_0\la\nrE{\vp_1-\vp_2}
\end{eqnarray*}
for all $[x_i,y_i] \in A(t_i,\vp_i),$  $i=1,2, \ \vp_i\in E.$
\end{assu}

\begin{assu} \label{IVFDE_2}
There exist bounded and Lipschitz continuous functions $g,h,k:\jz \to X, $ a constant $K_0 >0,$ and $L_1,L_2: \rep \lra \rep$ continuous and monotone non decreasing, 
such that for $\la >0$ and $t_1,t_2\in \jz$ we have
\begin{eqnarray*}
\lefteqn{\nrm{x_1-x_2}} \\
&\le& \nrm{x_1-x_2 -\la(y_1-y_2)}+\la \nrm{h(t_1)-h(t_2)}L_1(\nrm{x_2}) \\
&&+ \la \nrm{g(t_1)-g(t_2)}\nrm{y_2}+\la \nrm{k(t_1)-k(t_2)}L_2(\nrm{\vp_2})+K_0\la\nrE{\vp_1-\vp_2}
\end{eqnarray*}
for all $[x_i,y_i] \in A(t_i,\vp_i),$  $i=1,2, \ \vp_i\in E.$
\end{assu}

For a dissipative operator $A\subset X\times X$ we have, 
$$ 
J_{\la}x=(I-\la A)^{-1}x \mbox{ and } A_{\la}x=\frac{1}{\la}(I-J_{\la})x.
$$
With the above we define,
$$
\btr{Ax}:=\lim_{\la\to 0}\nrm{A_{\la}x} \mbox{ and } \hat{D}:=\bk{x\in X :\btr{Ax}<\infty }.
$$

Due to Assumption \ref{IVFDE_2} we have for given $\vp_1,\vp_2 \in E$, $t,s \in [0,T],$ and $ x\in \hat{D}(A(t,\vp_1)$ 
\begin{eqnarray*}
\lefteqn{\btr{A(s,\vp_2)x}}\\
&\le& \btr{A(t,\vp_1)x}+\nrm{h(s)-h(t)}L(\nrm{x})+\nrm{g(s)-g(t)}\btr{A(t,\vp_1)x} \\
&& + \nrm{k(s)-k(t)}L_2(\nrm{\vp}) +K_0\nrm{\vp_1-\vp_2}_E.
\end{eqnarray*}
A similar inequality comes with Assumption \ref{IVFDE_2_og}. In consequence, $\hat{D}(A(t,\vp))=\hat{D_{\vp}}.$

\begin{remk}\label{A_t_vp_uniformly_bd}
We have 
$$ \hat{D_{\vp_1}} = \hat{D}_{\vp_2} \mbox { } \forall \vp_1,\vp_2\in E.$$ 
Moreover, if $x\in \hat{D}_{\vp_0}$ for some $\vp_0\in E$ and $F\subset E$ is bounded we find some $K>0,$ such that 
$$
\sup_{t\in\jz, \vp\in F} \btr{A(t,\vp)x} \le K.
$$
In view of the previous observations we define $\hat{D}=\hat{D}_{\vp}.$ Thus we have 
$$
\overline{D_{\vp}}\subset \overline{\hat{D}_{ \vp}}\subset \overline{\hat{D}}=:\overline{D}.
$$
\end{remk}

As we consider $A(t,u_t)+\om I$ we need the perturbed control inequality  of Assumption       
\ref{IVFDE_2_og} and  \ref{IVFDE_2}. This is computed similar to \cite[pp. 1056-1057]{Kreulichevo} and leads with
$$\Funk{h^{\om}}{\jz}{\fk{ X\times X,\nrm{\cdot}_1}}{t}{(h(t),\btr{\om}g(t))}
$$ and $L_1^{\om}=L(t)+t$ in case Assumption \ref{IVFDE_2} to the modified inequality: 

\begin{eqnarray} \label{perturbed_control}
\lefteqn{\nrm{x_1-x_2}} \nonumber  \\
&\le& \nrm{x_1-x_2 -\la(y_1-y_2)}+\la \nrm{h^{\om}(t_1)-h^{\om}(t_2)}L_1^{\om}(\nrm{x_2}) \\
&&+ \la \nrm{g(t_1)-g(t_2)}\nrm{y_2}+ \nrm{k(t_1)-k(t_2)}L_2(\nrm{\vp_2}) \nonumber \\
&&+K_0\la\nrE{\vp_1-\vp_2} +\la\om\nrm{x_1-x_2}, \nonumber
\end{eqnarray}
for $t_i\in \jz,$ $\vp_i\in E,$ and  $(x_i,y_i) \in A(t_i,\vp_i).$
In the case of Assumption \ref{IVFDE_2_og} we have,
\begin{eqnarray} \label{perturbed_control_og}
\lefteqn{\nrm{x_1-x_2}} \nonumber  \\
&\le& \nrm{x_1-x_2 -\la(y_1-y_2)}+\la \nrm{h^{\om}(t_1)-h^{\om}(t_2)}L_1^{\om}(\nrm{x_2}) \\
&&+ \nrm{k(t_1)-k(t_2)}L_2(\nrm{\vp_2})+K_0\la\nrE{\vp_1-\vp_2} +\la\om\nrm{x_1-x_2}. \nonumber
\end{eqnarray}
for $t_i\in \jz,$ $\vp_i\in E,$ and  $(x_i,y_i) \in A(t_i,\vp_i).$
Throughout this study we define for $t\in\jz,$ $x\in X$ and $\vp\in E$,
\begin{eqnarray*}
J_{\la}^{\om}(t,\vp)x&:=&\fk{I-\la (A(t,\vp)+\om I)}^{-1}x\\
&=&\fk{I-\frac{\la}{1-\la\om} (A(t,\vp)}^{-1}\fk{\frac{1}{1-\la\om}x} \\
&=&J_{\frac{\la}{1-\la\om}}(t,\vp)\fk{\frac{1}{1-\la\om}x},
\end{eqnarray*}
for all $0<\la<\frac{1}{\om}.$

\section{Recursion}
The proof of existence is split into three main steps, the initialization  of a recursion $n=1,$  its step from $n\to n+1,$ for every small $\la,$ and finally the computation of the double limit $\lim_{n\to\infty}\lim_{\la\to 0}.$ To approximate the solution the method provided in \cite{Kreulichevo} is used. For given 
$\vp \in E,$ the initial history, let

$$
(\partial_t)_{\la}u(t):=\frac{1}{\la}\fk{u(t)-\vp(0)-\frac{1}{\la}\int_0^t\nres{\tau}(u(t-\tau)-\vp(0))d\tau}.
$$
Additionally, we use an approximation of the history state $\vp.$ For given $\psi\in Y$ we define,
$$
\vp_{\psi,\la}(t):=\left\{ \begin{array}{rcl}
					-\frac{t}{\la}\vp(-\la)+(1+\frac{t}{\la})\Jlw(0,(\psi)_0)\vp(0) &:& -\la\le t \le 0 \\
					\vp(t) &:& t<-\lambda.
					\end{array}
					\right.
$$

\begin{remk} \label{vp-la-minus-vp}
Let $\vp\in E,$ $\vp(0)\in\overline{D},$ and $\psi\in Y,$ then
\begin{enumerate}
\item $ \lim_{\la\to 0}\vp_{\psi,\la}=\vp $ in $E$
\item If $F\subset E$ is bounded, then $\bk{\vp_{\psi,\la}}_{\psi\in F, \la >0}$ is bounded as well.
\item If $F\subset E$ is bounded and $\vp$ is Lipschitz and $\vp(0)\in \hat{D},$ then $\bk{\vp_{\psi,\la}}_{\psi\in F,\la >0}$ is equi-Lipschitz.
\end{enumerate}
\end{remk}
\begin{proof} For given $\vp\in E$ and $\psi\in Y$ we only have to look at $-\la\le t\le0,$ which leads to
\begin{eqnarray*}
\nrm{\vp_{\psi,\la}(t)-\vp(t)}
&\le& \frac{-t}{\la}\nrm{\vp(-\la)-\vp(t)}+\fk{1+\frac{t}{\la}}\nrm{J^{\om}_{\la}(0,\psi_0)\vp(0)-\vp(0)} \\
&&+\fk{1+\frac{t}{\la}}\sup_{-\la\le t\le 0}\nrm{\vp(0)-\vp(t)}.
\end{eqnarray*}
The uniform continuity of $\vp$ and the Assumption \ref{IVFDE_1}, the m-dissipativeness serves for the proof.
For the second claim if $\phi \in E, $ $x\in\hat{D}$ fix, and $-\la\le t\le 0$  then for $\psi\in F $ we have
\begin{eqnarray*}
\lefteqn{\nrm{\vp_{\psi,\la}(t)-\vp(t)}} \\
&\le& \frac{-t}{\la}\nrm{\vp(-\la)-\vp(t)}+\fk{1+\frac{t}{\la}}\nrm{J^{\om}_{\la}(0,\psi_0)\vp(0)-J^{\om}_{\la}(0,\phi_0)\vp(0)} \\
&&+\fk{1+\frac{t}{\la}} \fk{ \nrm{J^{\om}_{\la}(0,\phi_0)\vp(0)-J^{\om}_{\la}(0,\phi_0)x} +\nrm{J^{\om}_{\la}(0,\phi_0)x-x}+\nrm{x-\vp(t)}} \\
&\le& \nrm{\vp(-\la)-\vp(t)} +\nrE{\psi-\phi}+\nrm{\vp(0)-x}+\la \btr{A(0,\phi_0)x} +\nrm{x-\vp(t)}.
\end{eqnarray*}

As $\vp$ is assumed to be Lipschitz it remains to to consider $-\la\le t,s\le 0.$ Thus we have
\begin{eqnarray*}
\lefteqn{\nrm{-\frac{t}{\la}\vp(-\la)+\fk{1+\frac{t}{\la}}J_{\la}^{\om}(0,\psi_0)\vp(0)
+\frac{s}{\la}\vp{-\la}-\fk{1+\frac{s}{\la}}J_{\la}^{\om}(0,\psi_0)\vp(0)} } \\
&=& \left\|-\frac{t}{\la}(\vp(-\la)-\vp(0))+\fk{1+\frac{t}{\la}}(J_{\la}^{\om}(0,\psi_0)\vp(0)-\vp(0)) \right. \\
&&  \quad \quad \left. +\frac{s}{\la}(\vp(-\la)-\vp(0))-\fk{1+\frac{s}{\la}}(J_{\la}^{\om}(0,\psi_0)\vp(0) -\vp(0) )  \right\| \\
&\le& \nrm{\frac{t-s}{\la}(\vp(-\la)-\vp(0))+\frac{t-s}{\la}\fk{J_{\la}^{\om}(0,\psi_0)\vp(0) -\vp(0) }} \\
&\le& \btr{t-s}(L_{\vp}+\btr{A(0,\psi_0)\vp(0)}).
\end{eqnarray*}
Now apply Remark \ref{A_t_vp_uniformly_bd}.
\end{proof}

By the above definition with a given $\psi\in Y$ we are able to define a recursion for the approximations.

\begin{recursion} \label{recursion}
Let $\psi\in Y.$ 
\begin{description} 
\item{$n=1$:} $u_{1,\la}$ is the solution to
\begin{equation} \label{recursion_start}
\left\{
\begin{array}{rcll} 
\fk{\partial_t}_{\la} u_{1,\la}(t) &\in & A(t,(\psi)_t)u_{1,\la}(t) +\om u_{1,\la}(t) &:t\in \jz\\
(u_{1,\la})_{|\ji} &=& \vp_{\psi,\la}. & 
\end{array}
\right.
\end{equation}
\item{$n\to n+1$:} If $\bk{u_{n,\la}}_{\la>0} \subset Y$ is the solution to the n-th equation we define $u_{n+1,\la}$ to be the solution to:
\begin{equation} \label{recursion_step} 
\left\{
\begin{array}{rcll} 
\fk{\partial_t}_{\la} u_{n+1,\la}(t) &\in & A(t,(u_{n,\la})_t)u_{n+1,\la}(t) +\om u_{n+1,\la}(t)&:t\in \jz\\
(u_{n+1,\la})_{|\ji} &=& \vp_{u_{n,\la},\la}. &
\end{array}
\right.
\end{equation}
\end{description}
\end{recursion}

\section{Existence on a Bounded Interval}

The idea is to apply the Banach Fixpoint Principle on $Y$ for a forthcoming iteration. We start with the following proposition.

\begin{pro} \label{T-la-defined}
Let $\jz=[0,T],$ $\vp\in E,$ $\psi \in Y$, and $A(t,\psi_t)$ satisfy Assumption \ref{IVFDE_1} and  either Assumption \ref{IVFDE_2_og}, or Assumption \ref{IVFDE_2}. Then $\bk{x \mapsto J^{\om}_{\la}(t,\psi_t)x}$ is Lipschitz, and
$$
F(u)(t):=\Jlw(t,(\psi)_t)\fk{\nres{t}\vp(0)+\frac{1}{\la}\int_0^t\nres{\tau}u(t-\tau)d\tau}
$$
has a fix point $(u_{1,\la}^+)_{|\jz}\in BUC(\jz,X).$ Moreover, the fix point satisfies
$$
u_{1,\la}^+(0)=\Jlw(0,\psi_0)\vp(0)=\vp_{\psi,\la}(0)\in\hat{D}_{\psi_{|\ji}}=\hat{D}.
$$
\end{pro}

\begin{proof}
Let $\psi \in Y$, then either from the inequalities (\ref{perturbed_control}) or (\ref{perturbed_control_og}) we obtain  $\bk{x \mapsto J^{\om}_{\la}(t,\psi_t)x}$ is Lipschitz with $\frac{1}{1-\la\om}.$ Thus, we are in the situation of \cite[Lemma 3.1, p. 1063]{Kreulichevo}, which implies that
$$
F(u)(t):=\Jlw(t,(\psi)_t)\fk{\nres{t}\vp(0)+\frac{1}{\la}\int_0^t\nres{\tau}u(t-\tau)d\tau}
$$
has a fix point in $BUC(\jz,X).$ The fix point equation lead to the additional claim.
\end{proof}

By the above observations we are able to define the solution operator:
\begin{equation}\label{T_la_def}
\begin{array}{ccccc}
\Tlp & : & Y & \lra & Y \\
      &   & \psi & \lmt & \left\{ \begin{array}{rcl}
		                    u_{1,\la}^+  &:& t\in\jz \\
		                    \vp_{\psi,\la} &:& t \in \ji.
	                       \end{array}\right.                                                
\end{array} 
\end{equation}

We define 
$$u_{n,\la}:=\Tlp^n\psi.$$

In the next step we show how an iterative use of \cite{Kreulichevo} applies to approximate the problem (\ref{gFDEonInterval}).

\begin{lem}\label{iteration-possible} 

Let $\jz=[0,T],$  Assumption \ref{IVFDE_1}, and either Assumption \ref{IVFDE_2_og}, or Assumption \ref{IVFDE_2} hold. If $\psi \in Y$  is Lipschitz and $\vp \in E$, with $\vp(0) \in \hat{D},$ then: 
 \begin{enumerate}
 \item[a)]$\nrm{\Tlp\psi(t)-\vp(0)}\le K^{\prime}(\la+t)$ for all $t\in\jz$ and some $K^{\prime}\ge 0.$ Consequently, $\bk{\Tlp\psi}_{\la>0}$ is uniformly bounded on $[0,T].$
 \item[b)] $\Tlp\psi(0)\in\hat{D}$ for small $\la.$
 \item[c)] $\bk{\Tlp\psi}$ is equi Lipschitz on $\jz$ for small $\la$ in the case of  Assumption \ref{IVFDE_2}
 \item[d)]$\Tlp\psi$ is uniformly equicontinuous on $\jz$ for small $\la$ in the case of  Assumption \ref{IVFDE_2_og}
 \item[e)]there exists $u_1 \in Y$ s.t $\lim_{\la\to 0} \Tlp\psi \to u_1$ uniformly on $\ji \cup [0,T].$ 
 \end{enumerate}
\end{lem}

\begin{proof}
We start with a),b),c) and restrict the proof to Assumption \ref{IVFDE_2}. The Assumption \ref{IVFDE_2} and the functions
$$
\Funk{\tilde{h}}{[0,T]}{(X\times X \times E,\nrm{\cdot}_1)}{t}{(h^{\om}(t),k(t),\psi_t),}
$$
and $L(\nrm{x},\nrm{\psi}_E):=L_1^{\om}(\nrm{x})+L_2(\nrm{\psi}_E)+K_0,$ imply that

\begin{eqnarray} \label{modified_control}
\nrm{x_1-x_2}
&\le & \nrm{x_1-x_2-\la(y_1-y_2)} +\la\nrm{\tilde{h}(t_1)-\tilde{h}(t_2)}L^{\om}(\nrm{x},\nrE{\psi}) \\
&& \la \nrm{g(t_1)-g(t_2)}\nrm{y_2}, \nonumber
\end{eqnarray}
for $(x_i,y_i)\in A(t_i,\psi_{t_i}).$ 
Hence, with $B(t):=A(t,\psi_t)$ for $t\in\jz$ we are in the situation of the proof of  \cite[Lemma 3.2 (1),(2), p. 1064]{Kreulichevo}, with 
$u_0=\vp(0)\in\hat{D},$ and obtain for  $u_{1,\la}=\Tlp\psi$ the claims  a), and c) on $\jz$.
The remaining claim $T_\la\psi(0)\in\hat{D}$ follows from,
$$
\Tlp\psi(0)=J_{\la}(0,(\psi)_0)\vp(0)\in D(A(0,\psi))\subset\hat{D}(A(0,\psi))\subset\hat{D}.
$$

For the proof of part e) recall that for $t\in\jz$ \cite[Theorem 2.9, p. 1058]{Kreulichevo} applies. For $t\in\ji$ we recall the Remark \ref{vp-la-minus-vp}.

\end{proof}

\begin{remk}\label{history_control}
Let $\vp_i \in E,$ $i=1,2,$ $z\in X$ and $t\in [0,T],$ then
\begin{equation}
\nrm{J_{\la}(t,\vp_1)z-J_{\la}(t,\vp_2)z}\le \frac{K\la}{1-\la\om}\nrE{\vp_1-\vp_2}.
\end{equation}
\end{remk}
\begin{proof}
Apply Assumption \ref{IVFDE_2} or \ref{IVFDE_2_og} with $A_{\la}(t,\vp_i)z \in A(t,\vp_i)J_{\la}(t,\vp_i)z.$
\end{proof}

\begin{lem} \label{ul-fixpoint}
Let for $\jz,$ Assumption \ref{IVFDE_1}, and either Assumption \ref{IVFDE_2_og} or \ref{IVFDE_2} hold.
Further, let $\vp(0)\in\overline{D},$ and $\psi\in Y,$ then $\Tlp$  satisfies the equation
\begin{equation} \label{psi-la-approx-gFDE}
\left\{
\begin{array}{rcll} 
\fk{\partial_t}_{\la} \Tlp\psi(t)&\in & A(t,(\psi)_t)\Tlp\psi(t)+\om \Tlp\psi(t)&:t\in\jz\\
\Tlp\psi_{|\ji} &=& \vp_{\psi,\la} &:t\in \ji.
\end{array}
\right.
\end{equation}
If $t\in\jz,$ and $\psi,\phi\in Y$ and two initial histories $\vp_1,\vp_2\in E$ we have,
\begin{eqnarray}\label{iterate_ineq}
\lefteqn{\nrm{T_{\la,\vp_1}\psi(t) -T_{\la,\vp_2}\phi(t)}}\nonumber \\
&\le& \frac{\la K_0}{1-\la\om}\nrm{\psi_t-\phi_t}_E 
+\frac{K_0}{1-\la\om}\int_0^t\exp\fk{\frac{\om}{1-\la\om}\tau}\nrm{\psi_{t-\tau}-\phi_{t-\tau}}_E d\tau \\
&& +\fk{\frac{1}{1-\la\om}\exp\fk{-\frac{t}{\la}} +\exp\fk{\frac{\om t }{1-\la\om}}}\nrm{\vp_1(0)-\vp_2(0)}. \nonumber
\end{eqnarray}

If $0\le t\le T,$ $\vp_1=\vp_2$ and $I(t):=\ji\cup [0,t]$ then
\begin{eqnarray} \label{iterate_ineq_2}
\lefteqn{\sup_{x\in I(t)}\nrm{\Tlp\psi(x) -\Tlp\phi(x)}}\nonumber \\
&\le&  \frac{\la K_0}{1-\la\om}\sup_{x\in I(t)}\nrm{\psi(x)-\phi(x)} \nonumber \\
&&+ \frac{K_0}{1-\la\om}\int_0^t\exp\fk{\frac{\om}{1-\la\om}\tau}\sup_{x\in I(t-\tau)}        \nrm{\psi(x)-\phi(x)} d\tau
\end{eqnarray}
\end{lem}

\begin{proof}
We restrict the proof to Assumption \ref{IVFDE_2}.
For the starting elements $\psi, \phi$ due to Proposition \ref{T-la-defined} the solution operator $T_\la$ is well defined. Let $\up{\la}{\vp_1,\psi}=T_{\la,\vp_1}\psi,$ and $\up{\la}{\vp_2,\phi}=T_{\la,\vp_2}\phi,$ the solution operators with two different initial histories $\vp_1,\vp_2.$ We start considering $t\in \jz,$ and apply Assumption \ref{IVFDE_2}:
\begin{eqnarray*}
\lefteqn{\nrm{T_{\la,\vp_1}\psi(t)-T_{\la,\vp_2}\phi(t)} =\nrm{\up{\la}{\vp_1,\psi}(t)-\up{\la}{\vp_1,\phi}(t)} }\\
& \le & \frac{1}{1-\la\om} \nrm{\up{\la}{\vp_1,\psi}(t)-\up{\la}{\vp_1,\phi}(t) -\la \fk{\dtl\up{\la}{\vp_1,\psi}(t)-\dtl\up{\la}{\vp_1,\phi}(t)}} \\
&&+\frac{\la}{1-\la\om} \nrE{\psi_t-\phi_t} \\
&\le& \frac{1}{\la(1-\la\om)}\int_0^t\nres{\tau}
   \nrm{\up{\la}{\psi}(t-\tau)-\up{\la}{\phi}(t-\tau)}d\tau \\
   && + \exp\fk{-\frac{t}{\la}}\nrm{\vp_1(0)-\vp_2(0)}+ \frac{\la K_0}{1-\la\om} \nrm{\psi_t-\phi_t}_E.
\end{eqnarray*}

The integral inequality \ref{s-t-integral-inequality} gives
\begin{eqnarray}  \label{t_greater_0}
\lefteqn{\nrm{\up{\la}{\vp_1,\psi}(t)-\up{\la}{\vp_2,\phi}(t)} } \nonumber \\
&\le& \frac{\la K_0}{1-\la\om}\nrm{\psi_t-\phi_t}_E 
+\frac{K_0}{1-\la\om}\int_0^t\exp\fk{\frac{\om}{1-\la\om}\tau}\nrm{\psi_{t-\tau}-\phi_{t-\tau}}_E d\tau\\
&& +\fk{\frac{1}{1-\la\om}\exp\fk{-\frac{t}{\la}} +\exp\fk{\frac{\om t}{1-\la\om}}}\nrm{\vp_1(0)-\vp_2(0)}. \nonumber
\end{eqnarray}
Now, consider $t\in\ji,$ and $\vp_1=\vp_2=\vp.$ Noting that $u_{\la,\vp,\psi}=u_{\la,\vp,\phi}$ for $t\le-\la,$ and using the  Assumption \ref{IVFDE_2} and $t\in [-\la,0]$ we have by the Remark \ref{history_control}:
\begin{eqnarray*}
\lefteqn{\nrm{\up{\la,\vp}{\psi}(t)-\up{\la,\vp}{\phi}(t)}}\\ 
&=&\nrm{-\frac{t}{\la}\vp(-\la)+\fk{1+\frac{t}{\la}}\Jl(0,\psi)\vp(0)-
\fk{\frac{-t}{\la}\vp(-\la )+\fk{1+\frac{t}{\la}}\Jl(0,\phi)\vp(0)}}\\
&=&\fk{1+\frac{t}{\la}}\nrm{\Jl(0,\psi)\vp(0)-\Jl(0,\phi)\vp(0)}\\
&\le&\la K_0\frac{\fk{1+\frac{t}{\la}}}{1-\la\om}\nrm{\psi-\phi}_E \le \la K_0\frac{1}{1-\la\om}\nrm{\psi-\phi}_E.
\end{eqnarray*}
Hence,
\begin{eqnarray} \label{E_inequality}
\nrm{\up{\la,\vp}{\psi}-\up{\la,\vp}{\phi}}_E
&\le&\frac{\fk{\la K_0}}{1-\la\om}\nrm{\psi-\phi}_E.
\end{eqnarray}
Defining for $n\in\za\cup \bk{0},$
$$
\Funk{f_n^{\la}}{\jz}{\rep}{t}{\sup_{\tau\in \ji\cup[0,t]}\nrm{T^n_{\la,\vp}\psi(\tau)-T^n_{\la,\vp}\phi(\tau)}},
$$
we have $\nrm{\psi_t-\phi_t}_E\le f_0^{\la}(t).$ The definition of
$$
\Funk{S_{\la}}{BUC(\jz)}{BUC(\jz)}{g}{\bk{t\mapsto \frac{K_0}{1-\la\om}\int_0^t\exp\fk{\frac{-\om s}{1-\la\om}}g(t-s)ds}},
$$ leads to
\begin{eqnarray} \label{iterating_ineq_start}
\nrm{\up{\la,\vp}{\psi}(t)-\up{\la,\vp}{\phi}(t)}&\le \frac{\la K_0}{1-\la\om}f_0^{\la}(t) +S_{\la}f_0^{\la}(t). 
\end{eqnarray}

Note that $f_0^{\la}$ is non-decreasing and positive, hence $S_{\la}f_0^{\la}$ is non-decreasing by Proposition \ref{non-deceasing-convolution}.
Thus, by the previous inequalities (\ref{t_greater_0}), (\ref{E_inequality})  and using $f_n^{\la}$ non-decreasing, we end up with the inequality (\ref{iterate_ineq_2})

\begin{eqnarray*}
f_1^{\la}(t)&\le \frac{\la K_0}{1-\la\om}f_0^{\la}(t) +S_{\la}f_0^{\la}(t). 
\end{eqnarray*}

\end{proof}
\begin{cor} \label{independence}
With the notion of the previous Lemma and proof we have for $T>0:$
\begin{equation} \label{iterative_inequality}
f_n^{\la}(T)\le \nrm{\fk{\frac{\la K_0}{1-\la\om}I+S_{\la}}^n}f_0^{\la}(T),
\end{equation}
and for the spectrum we have,
\begin{equation}\label{spectrum}
\sigma_{C[0.T]}\fk{\frac{\la K_0}{1-\la\om}I+S_{\la}}=\bk{\frac{\la K_0}{1-\la\om}}.
\end{equation}
\end{cor}
\begin{proof}
The claim (\ref{iterative_inequality}) comes with iterating (\ref{iterating_ineq_start}). The claim (\ref{spectrum}) is a consequence of  $S_{\la}$ quasi-nilpotent and the Spectral Mapping Theorem \cite[Thm. 10.28]{RudinFA}.
\end{proof}

Next we provide the methods to prove the existence of the limit $\la \to 0.$

\begin{lem} \label{induction_step_1}
Let for $\jz=[0,T]$ Assumption \ref{IVFDE_1} and either Assumption \ref{IVFDE_2_og} or Assumption \ref{IVFDE_2} hold. Further let $\vp\in E$  with $\vp(0)\in\overline{D},$ $\bk{\psi_{\la}}_{\la>0}\subset Y$ and $\psi\in Y.$ 
If
$$
\lim_{\la\to 0} \psi_{\la}=\psi,
$$
and $\ul\in Y$ is the solution to
\begin{equation} 
\left\{
\begin{array}{rcll} 
\fk{\partial_t}_{\la} \ul(t) &\in & A(t,(\psi_{\la})_t)\ul(t)+\om\ul(t)&:t\in\jz\\
(u_{\la})_{|\ji} &=& \vp_{\psi_{\la},\la} &:t\in \ji,
\end{array}
\right.
\end{equation}
and $v_{\la}\in Y$ is the solution to 
\begin{equation} 
\left\{
\begin{array}{rcll} 
\fk{\partial_t}_{\la} v_{\la}(t) &\in & A(t,(\psi)_t)v_{\la}(t)+\om v_{\la}(t)&:t\in\jz\\
(v_{\la})_{|\ji} &=& \vp_{\psi,\la} &:t\in \ji,
\end{array}
\right.
\end{equation}
then 
$$
\lim_{\la\to 0}\nrm{\ul-v_{\la}}_Y=0.
$$
\end{lem}
\begin{proof}
Apply inequality (\ref{iterate_ineq_2}) from Lemma \ref{ul-fixpoint}  with $\phi=\psi_{\la}$ and $\psi=\psi.$
\end{proof}

\begin{cor} \label{induction_step_2}

Under the conditions of the previous lemma, $\bk{u_{\la}}_{\la>0}$ and $\bk{v_{\la}}_{\la>0}$ are Cauchy in Y for $\la\to 0.$
\end{cor}
\begin{proof}
Using a uniform continuous extension of $\psi$ on $\re$, and a mollifier we find that the Lipschitz functions on $\ji\cup\jz$ are dense in $Y.$ Consider $\psi^{\ep}$ Lipschitz and the equation, 
\begin{equation} 
\left\{
\begin{array}{rcll} 
\fk{\partial_t}_{\la} v_{\la}^{\ep}(t) &\in & A(t,(\psi^{\ep})_t)v_{\la}^{\ep}(t)+\om v_{\la}^{\ep}(t)&:t\ge 0\\
(v_{\la}^{\ep})_{|\ji} &=& \vp_{\psi^{\ep},\la} &:t\in \ji.
\end{array}
\right.
\end{equation}
Applying Lemma \ref{iteration-possible} we find $\bk{v_{\la}^{\ep}}_{\la>0}$ for every $\psi^{\ep}$ arbitrary close to $\psi.$  
The use of (\ref{iterate_ineq}) from Lemma \ref{ul-fixpoint} and the triangle inequality gives
\begin{eqnarray*}
\nrm{v_{\la}(t)-v_{\mu}(t)} &\le & 
\nrm{v_{\la}(t)-v_{\la}^{\ep}(t)} + \nrm{v_{\la}^{\ep}(t)-v_{\mu}^{\ep}(t)}+\nrm{v_{\mu}^{\ep}(t)-v_{\mu}(t)} \\
&\le& \frac{\la K_0}{1-\la\om}\nrm{\psi_t-\psi^{\ep}_t}_E 
+\frac{K_0}{1-\la\om}\int_0^t\exp\fk{\frac{\om}{1-\la\om}\tau}\nrm{\psi_{t-\tau}-\psi^{\ep}_{t-\tau}}_E d\tau \\
&& + \nrm{v_{\la}^{\ep}(t)-v_{\mu}^{\ep}(t)} +\frac{\mu K_0}{1-\mu\om}\nrm{\psi_t-\psi^{\ep}_t}_E \\ 
&& +\frac{K_0}{1-\mu\om}\int_0^t\exp\fk{\frac{\om}{1-\mu\om}\tau}\nrm{\psi_{t-\tau}-\psi^{\ep}_{t-\tau}}_E d\tau. \\
\end{eqnarray*}
It remains to prove that $\bk{v_{\la}^{\ep}}_{\la>0} $ is Cauchy, as for small $\la>0$
\begin{eqnarray} \label{approx-vp}
\nrm{\vp_{\psi^{\ep},\la}(0)-\vp(0)}&=& \nrm{J_{\la}^{\om}(0,\psi_0^{\ep})\vp(0) - \vp(0)} \nonumber \\
&\le&\nrm{J_{\la}^{\om}(0,\psi_0^{\ep})\vp(0)-J_{\la}^{\om}(0,\psi_0)\vp(0)} +
 \nrm{J_{\la}^{\om}(0,\psi_0)\vp(0)-\vp(0)} \nonumber \\
 &\le& \nrE{\psi^{\ep}-\psi} + \nrm{J_{\la}^{\om}(0,\psi_0)\vp(0)-\vp(0)} \\
 &\le& \ep. \nonumber
\end{eqnarray}

For small $\la$ we consider for  an $\alpha >0$
\begin{equation} 
\left\{
\begin{array}{rcll} 
\fk{\partial_t}_{\la} w_{\la}^{\ep,\alpha}(t) &\in & A(t,(\psi^{\ep})_t)w_{\la}^{\ep,\alpha}(t)+\om w_{\la}^{\ep,\alpha}(t)&:t\in\jz\\
(w_{\la}^{\ep,\alpha})(0) &=& J_{\alpha}(0,\psi_0)\vp(0)\in \hat{D}.
\end{array}
\right.
\end{equation}
By Lemma \ref{iteration-possible} and  \cite[Thm 2.9, p. 1058]{Kreulichevo} $\bk{w^{\ep,\alpha}_{\la}}_{\la>0}$ is Cauchy when $\la\to 0.$ From Lemma \ref{ul-fixpoint}  we obtain with the pairs $(v^{\ep}_{\la}(t), J_{\la}^{\om}(0,\psi_0)\vp(0)),$ and $(w^{\ep,\alpha}_{\la}(t),\vp(0))$
\begin{eqnarray*}
\lefteqn{\nrm{v^{\ep}_{\la}(t)-w^{\ep,\alpha}_{\la}(t)}} \\
&\le&\frac{1}{1-\la\om} \exp(-t/\la)\nrm{J_{\la}^{\om}(0,\psi_0^{\ep})\vp(0)-J_{\alpha}^{\om}(0,\psi_0)\vp(0)}\\
&& +\exp\fk{\frac{\om t}{1-\la\om}}\nrm{J_{\la}^{\om}(0,\psi_0^{\ep})\vp(0)-J_{\alpha}^{\om}(0,\psi_0)\vp(0)} \\
&\le \ep
\end{eqnarray*}
when $\alpha,\la$ small. Which gives that $\bk{v_{\la}^{\ep}}_{\la>0} $ is Cauchy, consequently 
$\bk{v_{\la}}_{\la>0} $ and therefore $\bk{u_{\la}}_{\la>0} $ for $t\in \jz.$ For $t\in \ji,$ apply Remark \ref{vp-la-minus-vp} 
with $\psi=\psi^{\ep}.$
\end{proof}

By previous observations we are in the situation to do the induction.

\begin{lem} \label{iteration_convergent}
Let for $\jz=[0,T]$ Assumption \ref{IVFDE_1} and either Assumption \ref{IVFDE_2_og} or Assumption \ref{IVFDE_2} hold. Further let $\vp\in E$  with $\vp(0)\in\overline{D},$ and $\psi\in Y.$ 
If $u_{1,\la}$ is the solution to (\ref{recursion_start}) 
then 
$$
\lim_{\la \to 0}u_{1,\la}=u_1\in Y.
$$
If $\bk{u_{n,\la}}_{\la>0} \subset Y$ is the solution to the n-th step with 
$$\lim_{\la\to 0}u_{n,\la}=u_n \in Y,$$
and $u_{n+1,\la}$ is the solution to (\ref{recursion_step}),
then 
$$
\lim_{\la \to 0}u_{n+1,\la}=u_{n+1}\in Y.
$$

\end{lem}
\begin{proof} 
Apply Lemma \ref{iteration-possible} and Corollary \ref{induction_step_2} to the start of the induction, the induction step follows by Lemma \ref{induction_step_1} and Corollary \ref{induction_step_2}.
\end{proof}

We are ready to state the main result of this section on the existence of a solution to (\ref{gFDEonInterval}) for the finite and infinite delay case , and for arbitrary $\om\in \re.$
As we found a sequence of functions $\bk{u_n}_{n\in\za}$ it remains to prove their convergence in $Y,$ and the independence of the starting point $\psi$ of the recursion.

\begin{theo} \label{0-T-solution}
Let for $\jz=[0,T]$ Assumption \ref{IVFDE_1} and either Assumption \ref{IVFDE_2_og} or Assumption \ref{IVFDE_2} hold. Further let $\vp\in E$  with $\vp(0)\in\overline{D},$ and $\psi\in Y.$ 
The sequence $\bk{u_n}_{n\in\za}$ defined in Lemma \ref{iteration_convergent} is uniformly convergent on $\ji\cup\jz.$ As the solution is independent of the approximation we call this limit the solution to (\ref{gFDEonInterval}) on $\ji\cup\jz.$
\end{theo}

\begin{proof}
As $(u_n)_{|\ji}=\vp$ it remains to consider $t\in\jz.$
From (\ref{iterate_ineq}) we obtain with the definition of the operator $\Tlp$ given by (\ref{T_la_def}),
\begin{eqnarray*}
\lefteqn{\nrm{u_{\la,n+1}(t)-u_{\la,n}(t)}} \\
&=&\nrm{(\Tlp u_{\la,n})_t -(\Tlp u_{\la,n-1})_t} \\
&\le & \frac{\la K_0}{1-\la\om}\nrm{(u_{\la,n})_t-(u_{\la,n-1})_t}_E 
+\frac{K_eK_0}{1-\la\om}\int_0^t\nrm{(u_{\la,n})_{t-\tau}-(u_{\la,n-1})_{t-\tau}}_E d\tau. 
\end{eqnarray*}
By Lemma \ref{iteration_convergent} we may pass to $\la \to 0$ and the previous inequality becomes,
\begin{eqnarray*}
\nrm{u_{n+1}(t)-u_{n}(t)} &\le&  K_eK_0\int_0^t\nrm{(u_{n})_{t-\tau}-(u_{n-1})_{t-\tau}}_E d\tau \\
&=&K_eK_0\int_0^t\nrm{(u_{n})_{\tau}-(u_{n-1})_{\tau)}}_E d\tau.
\end{eqnarray*}
As the integral is non-decreasing we derive,
$$
\nrE{(u_{n+1})_t-(u_{n})_t} \le K_eK_0\int_0^t\nrm{(u_{n})_{t-\tau}-(u_{n-1})_{t-\tau}}_E  d\tau.
$$
From Lemma \ref{iteration-possible} we have $u_{1,\la}$ uniformly bounded and therefore $u_1$ is bounded.
Lemma \ref{ul-fixpoint} leads by the limit $\la\to 0$
\begin{eqnarray*}
\nrm{(u_2)_t-(u_1)_t}_E&\le & K_eK_0 \int_0^t \nrm{(u_1)_{t-\tau}-\psi_{t-\tau}}_E d\tau \\
&\le& K_eK_0 \int_0^t K d\tau \\
&=& K_eK_0 K t. \\
\end{eqnarray*}
Iterating the inequality we find,
\begin{eqnarray*}
\nrm{(u_{n+1})_t-(u_n)_t}_E&\le&\frac{(K_eK_0K T)^{n}}{n!}
\end{eqnarray*}
which yields $\bk{u_n}_{n\in\za}$ uniformly Cauchy, and by the completeness of $Y$ we find a limit. Next we show that the limit is independent of the starting point $\psi.$ For two starting points $\psi,\phi\in Y$ we have by inequality (\ref{iterative_inequality}) of corollary \ref{independence},
$$
\nrm{T^n_{\la,\vp}\psi-T^n_{\la,\vp}\phi}_Y\le \nrm{\fk{\frac{\la K_0}{1-\la\om}I+S_{\la}}^n}\nrm{\psi-\phi}_Y.
$$
As we may pass to $\la\to 0$ we conclude with $T^n_{\la,\vp}\psi\to u_n,$ and $T^n_{\la,\vp}\phi\to v_n,$
$$
\nrm{u_n-v_n}_Y\le \nrm{S_{0}^n}\nrm{\psi-\phi}_Y.
$$
Using $S_0$ quasi-nilpotent we finish the proof when passing to $n\to \infty.$
\end{proof}

\section{Half-Line Functional Differential Equations} \label{Half-Line-FDEs}
In this section we show that under sufficient conditions on $\om$ we can conclude the convergence on $\jz=\rep$ of the approximation stated in Recursion \ref{recursion}. Moreover, some results on the asymptotic behavior on the half line are given.

\begin{lem}\label{iteration-possible_halfline}
Let for $\jz=\rep,$  Assumption \ref{IVFDE_1}, and either Assumption \ref{IVFDE_2_og} with $\om <0$, or Assumption \ref{IVFDE_2} with $L_g<-\om$ hold. If $\psi \in Y$  Lipschitz and $\vp \in E$ with $\vp(0) \in \overline{D},$ then 
\begin{enumerate}
 \item[a)] $\bk{\Tlp\psi}_{\la>0}$ is uniformly bounded on $\jz.$
  \item[b)] $\Tlp\psi$ is equi Lipschitz on $\jz$ in the case of  Assumption \ref{IVFDE_2} and $L_g<-\om:$
 \item[c)]$\Tlp\psi$ is uniformly equicontinuous on $\jz$ in the case of  Assumption \ref{IVFDE_2_og} and $\om <0.$
 \item[d)]there exists $u_1 \in Y$ s.t $\lim_{\la\to 0} \Tlp\psi \to u_1$ uniformly on $\ji \cup \jz.$ 
 \end{enumerate}
\end{lem}

\begin{proof} 
Similar to the proof of Lemma \ref{iteration-possible} we derive from Assumption \ref{IVFDE_2} the needed inequality for $B(t):=A(t,\psi_t).$ By \cite[Proof of Lemma 3.2 equation(16) p. 1064, and Corollary 3.3]{Kreulichevo} we obtain a). 
To verify c) apply \cite[Corollary 3.4, p. 1068]{Kreulichevo}, e) comes with \cite[Theorem 2.17, p.1061]{Kreulichevo}. d) is a consequence of \cite[Corollary 3.6, p. 1069]{Kreulichevo}.
\end{proof}

\begin{lem} \label{induction_step_1_halfline}
Let for $\jz=\rep$ Assumption \ref{IVFDE_1} and either Assumption \ref{IVFDE_2_og} with $\om <0$ or Assumption \ref{IVFDE_2} with $L_g<-\om$ hold. Further, let $\vp \in E$  with $\vp(0)\in\overline{D},$ $\bk{\psi_{\la}}_{\la>0}\subset Y$ and $\psi\in Y.$ 
If
$$
\lim_{\la\to 0} \psi_{\la}=\psi, \mbox{ in } Y,
$$
$\ul\in Y$ the solution to
\begin{equation} 
\left\{
\begin{array}{rcll} 
\fk{\partial_t}_{\la} \ul(t) &\in & A(t,(\psi_{\la})_t)\ul(t)+\om \ul(t)&:t\in\jz\\
(u_{\la})_{|\ji} &=& \vp_{\psi_{\la},\la}, &
\end{array}
\right.
\end{equation}
and $v_{\la}\in Y$ the solution to 
\begin{equation} 
\left\{
\begin{array}{rcll} 
\fk{\partial_t}_{\la} v_{\la}(t) &\in & A(t,(\psi)_t)v_{\la}(t)+\om v_{\la}(t)&:t\in\jz 0\\
(v_{\la})_{|\ji} &=& \vp_{\psi,\la}, &
\end{array}
\right.
\end{equation}
then 
$$
\lim_{\la\to 0}\nrm{\ul-v_{\la}}_Y=0.
$$
\end{lem}

\begin{proof}
Apply inequality (\ref{iterate_ineq_2}) from Lemma \ref{ul-fixpoint}  with $\phi=\psi_{\la}$ and $\psi=\psi.$
\end{proof}

\begin{cor} \label{induction_step_2_halfline}
Under the conditions of the previous lemma $\bk{v_{\la}}_{\la>0},$ and $\bk{\ul}_{\la>0}$ are Cauchy in Y for $\la\to 0.$
\end{cor}

\begin{proof}
Using a uniform continuous extension of $\psi$ on $\re$, and a mollifier we find that the Lipschitz functions on $\ji\cup\jz$ are dense in $Y.$ Consider $\psi^{\ep}$ Lipschitz and the equation 
for 
\begin{equation} 
\left\{
\begin{array}{rcll} 
\fk{\partial_t}_{\la} v_{\la}^{\ep}(t) &\in & A(t,(\psi^{\ep})_t)v_{\la}^{\ep}(t)+\om v_{\la}^{\ep}(t)&:t\in \jz\\
(v_{\la}^{\ep})_{|\ji} &=& \vp_{\psi^{\ep},\la}.&
\end{array}
\right.
\end{equation}
The use of inequality (\ref{iterate_ineq}) from Lemma \ref{ul-fixpoint} and the triangle inequality gives
\begin{eqnarray*}
\nrm{v_{\la}(t)-v_{\mu}(t)} &\le & 
\nrm{v_{\la}(t)-v_{\la}^{\ep}(t)} + \nrm{v_{\la}^{\ep}(t)-v_{\mu}^{\ep}(t)}+\nrm{v_{\mu}^{\ep}(t)-v_{\mu}(t)} \\
&\le& \frac{\la K_0}{1-\la\om}\nrm{\psi_t-\psi^{\ep}_t}_E 
+\frac{K_0}{1-\la\om}\int_0^t\exp\fk{\frac{\om}{1-\la\om}\tau}\nrm{\psi_{t-\tau}-\psi^{\ep}_{t-\tau}}_E d\tau \\
&& + \nrm{v_{\la}^{\ep}(t)-v_{\mu}^{\ep}(t)} +\frac{\mu K_0}{1-\mu\om}\nrm{\psi_t-\psi^{\ep}_t}_E \\
&& + \frac{K_0}{1-\mu\om}\int_0^t\exp\fk{\frac{\om}{1-\mu\om}\tau}\nrm{\psi_{t-\tau}-\psi^{\ep}_{t-\tau}}_E d\tau. \\
\end{eqnarray*}
It remains to prove that $\bk{v_{\la}^{\ep}}_{\la>0} $ is Cauchy.
Similar to (\ref{approx-vp}) we get
$$
\vp_{\psi^{\ep},\la}(0)=J_{\la}^{\om}(0,\psi_0)\vp(0) \to \vp(0)
$$
when $\la\to 0,$ we consider
\begin{equation} 
\left\{
\begin{array}{rcll} 
\fk{\partial_t}_{\la} w_{\la}^{\ep}(t) &\in & A(t,(\psi^{\ep})_t)w_{\la}^{\ep}(t)+\om w_{\la}^{\ep}(t)&:t\in\jz\\
(w_{\la}^{\ep})_{|\ji} &=& \vp(0).
\end{array}
\right.
\end{equation}
By Lemma \ref{iteration-possible_halfline} and \cite[Thm 2.13 (2), p. 1060]{Kreulichevo} $\bk{w^{\ep}_{\la}}_{\la>0}$ is Cauchy when $\la\to 0.$ From Lemma \ref{ul-fixpoint}  we obtain with the pairs $(v^{\ep}_{\la}(t), J_{\la}^{\om}(0,\psi_0)\vp(0)),$ and $(w^{\ep}_{\la}(t),\vp(0))$
\begin{eqnarray*}
\lefteqn{\nrm{v^{\ep}_{\la}(t)-w^{\ep}_{\la}(t)}} \\
&\le&\frac{1}{1-\la\om} \exp(-t/\la)\nrm{J_{\la}^{\om}(0,\psi_0^{\ep})\vp(0)-\vp(0)} +\exp\fk{\frac{\om t}{1-\la\om}}\nrm{J_{\la}^{\om}(0,\psi_0^{\ep})\vp(0)-\vp(0)} \\
&\le& \ep.
\end{eqnarray*}
for small $\la>0.$ Which gives that $\bk{v_{\la}}_{\la>0} $ is Cauchy, and therefore $\bk{u_{\la}}_{\la>0} $ for $t\in \jz.$ For $t\in\ji$ recall Remark \ref{vp-la-minus-vp}.
Which implies that $\bk{v_{\la}}_{\la>0}$ is Cauchy and therefore $\bk{u_{\la}}_{\la>0}.$
\end{proof}

\begin{lem} \label{iteration_convergent_halfline}
Let for $\jz=\rep$ Assumption \ref{IVFDE_1} and either Assumption \ref{IVFDE_2_og} with $\om <0$ or Assumption \ref{IVFDE_2} with $L_g<-\om$ hold. Further, let $\vp\in E$  with $\vp(0)\in\overline{D},$ and $\psi\in Y.$ 
If $u_{1,\la}$ is the solution to (\ref{recursion_start}) 
then 
$$
\lim_{\la \to 0}u_{1,\la}=u_1\in Y.
$$
If $\bk{u_{n,\la}}_{\la>0} \subset Y$ is the solution to the n-th step with 
$$\lim_{\la\to 0}u_{n,\la}=u_n \in Y,$$
and $u_{n+1,\la}$ is the solution to (\ref{recursion_step}),
then 
$$
\lim_{\la \to 0}u_{n+1,\la}=u_{n+1}\in Y.
$$
\end{lem}

\begin{proof} 
Apply Lemma \ref{iteration-possible_halfline} and Corollary \ref{induction_step_2_halfline} to the start the induction, the induction step follows by Lemma \ref{induction_step_1_halfline} and Corollary \ref{induction_step_2_halfline}.
\end{proof}

\begin{theo} \label{halfline-uniform-convergent}
Let for $\jz=\rep$ Assumption \ref{IVFDE_1} and either Assumption \ref{IVFDE_2_og} with $K_0<-\om,$ or Assumption \ref{IVFDE_2} with $\max\bk{K_0,L_g}<-\om$ hold. Further let $\vp \in E$  with $\vp(0)\in\overline{D},$ and $\psi\in Y.$ 
The sequence defined in the Lemma \ref{iteration_convergent_halfline} is uniformly convergent on $\ji\cup\jz.$
As the limit is independent of the starting point $\psi,$ we call it the solution to (\ref{gFDEonInterval}) on $\ji\cup\jz.$
\end{theo}

\begin{proof}
As $(u_n)_{|\ji}=\vp$ it remains to consider $t\in\jz.$ Let $t\in\jz$, and $n\ge 1.$ The we have by 
Lemma \ref{ul-fixpoint} and inequality (\ref{iterate_ineq_2}) that

\begin{eqnarray*} 
\lefteqn{\sup_{x\in I(t)}\nrm{T^{n+1}_{\la,\vp}\psi(x) -T^{n}_{\la,\vp}\phi(x)}}\nonumber \\
&\le&  \frac{\la K_0}{1-\la\om}\sup_{x\in I(t)}\nrm{T^n_{\la,\vp}\psi(x)-T^{n-1}_{\la,\vp}\phi(x)} \nonumber \\
&&+ \frac{K_0}{1-\la\om}\int_0^t\exp\fk{\frac{\om}{1-\la\om}\tau}\sup_{x\in I(t-\tau)}        \nrm{T^n_{\la,\vp}\psi(x)-T^{n-1}_{\la,\vp}\phi(x)} d\tau.
\end{eqnarray*}
As we may pass $\la\to 0$, we obtain,
\begin{eqnarray*}
\sup_{x\in I(t)}\nrm{u_{n+1}(x)-u_n(x)} &\le&K_0 \int_0^t\exp\fk{\om \tau}\sup_{x\in I(t-\tau)}\nrm{u_n(x)-u_{n-1}(x)}d\tau \\
&\le& \frac{K_0}{\btr{\om}}(1-\exp(\om t)\sup_{x\in I(t)}\nrm{u_n(x)-u_{n-1}(x)}.
\end{eqnarray*}
Consequently we have
$$
\nrm{u_{n+1}-u_n}_Y \le \fk{\frac{K_0}{-\om}}^n \nrm{u_1-\psi}_Y
$$
and $K_0<-\om$ gives $\bk{u_n}_{n\in\za}$ is Cauchy in $Y.$ Thus it remains to prove the independence on the starting point $\psi.$
For two starting points $\psi,\phi\in Y$ we have by inequality (\ref{iterative_inequality}) of corollary \ref{independence},
$$
\nrm{T^n_{\la,\vp}\psi-T^n_{\la,\vp}\phi}_Y\le \nrm{\fk{\frac{\la K_0}{1-\la\om}I+S_{\la}}^n}\nrm{\psi-\phi}_Y.
$$
As we may pass $\la\to 0$ we conclude with $T^n_{\la,\vp}\psi\to u_n,$ and $T^n_{\la,\vp}\phi\to v_n,$
$$
\nrm{u_n-v_n}_Y\le \nrm{S_{0}^n}\nrm{\psi-\phi}_Y.
$$
Using $\nrm{S_0^n}\le (\frac{K}{-\om})^n$ we finish the proof when passing to $n\to \infty.$
\end{proof}

\section{Asymptotic Behavior of the Solution}

\begin{theo}
For $\jz=\rep$ let Assumption \ref{IVFDE_1} and either Assumption \ref{IVFDE_2_og} with $K_0<-\om,$ or Assumption \ref{IVFDE_2} with $\max\bk{K_0,L_g}<-\om$ hold. Furthermore, let $\vp \in E$ with $\vp(0)\in\overline{D},$ and let $Y=\bk{f \in BUC(\ji\cup\jz): \lim_{t\to\infty}f(t)=0}.$ If 
\begin{equation} \label{C_0 condition}
\bk{t \mapsto J_{\la}(t,\psi_t)0} \in Y \ \mbox{ for all } \psi\in Y, 
\end{equation}
then the solution $u$ to (\ref{gFDEonInterval}) is an element of $Y.$
\end{theo}

\begin{proof}
As we obtain $ u$ as an iterated uniform limit,
$$
u=\lim_{n\to\infty}\lim_{\la\to 0}u_{n,\la}
$$
it remains to show that $u_{n,\la}\in Y.$ 
Using (\ref{C_0 condition}) implies that $\bk{t\mapsto \Jl^{\om}(t,\psi_t)0}\in Y$, Thus for $u\in Y $ we have
\begin{eqnarray*}
\lefteqn{\nrm{\Jl^{\om}(t,(\psi)_t)\fk{\nres{t}\vp(0)+\frac{1}{\la}\int_0^t\nres{\tau}u(t-\tau)d\tau}-
\Jl^{\om}(t,\psi_t)0}} \\
&\le& \frac{1}{1-\la\om}\nrm{\fk{\nres{t}\vp(0)+\frac{1}{\la}\int_0^t\nres{\tau}u(t-\tau)d\tau}}\\
&& \to 0, \mbox{ when } t\to \infty
\end{eqnarray*}
for all $u\in Y.$
Thus,
$$
\Funk{F}{Y}{Y}{u}{\bk{t\mapsto \Jl^{\om}(t,(\psi)_t)\fk{\nres{t}\vp(0)+\frac{1}{\la}\int_0^t\nres{\tau}u(t-\tau)d\tau}}}
$$
Hence,$\bk{u_{1,\la}}_{\la>0} \subset Y,$  and 
therefore $\bk{u_{n,\la}}_{\la>0,n\in\za},$ and the proof is finished.
\end{proof}

As we are interested in several types of almost periodicity in the next theorem it is shown how to obtain solutions in a general closed and translation invariant subspace $Z\subset BUC(\rep,X).$
\begin{theo} \label{half-line-asymptotics}
Let for $\jz=\rep$ Assumption \ref{IVFDE_1} and either Assumption \ref{IVFDE_2_og} with $\om <0$ or Assumption \ref{IVFDE_2} with $\max\bk{K_0,L_g}<-\om$ hold. Further let $\vp\in E$ with $\vp(0)\in\overline{D}.$ 
If for a closed and translation invariant subspace $Z$ with $C_0(\rep,X)\subset Z \subset BUC(\rep,X),$
$$
Y=\bk{f \in BUC(\ji\cup\jz,X): f_{|\rep}\in Z}
$$
and 
$$
\bk{t \mapsto J_{\la}^{\om}(t,\psi_t)f(t)} \in Z \ \mbox{ for all } \psi\in Y, f\in Z
$$
then the solution $u$ to (\ref{gFDEonInterval}) is an element of $Y.$
\end{theo}
\begin{proof}
For given $\psi \in Y $ and $u\in Z$ we have 
$$
\nres{t}\vp(0)+\frac{1}{\la}\int_0^t\nres{\tau}u(t-\tau)d\tau \in Z
$$
and consequently
$$
\bk{t\mapsto \Jl^{\om}(t,\psi_t)\fk{\nres{t}\vp(0)+\frac{1}{\la}\int_0^t\nres{\tau}u(t-\tau)d\tau}} \in Z.
$$
Hence, we can define the fixpoint mapping,
$$
\Funk{F_{0,\la}}{Z}{Z}{u}{\bk{t\mapsto \Jl^{\om}(t,(\psi)_t)\fk{\nres{t}\vp(0)+\frac{1}{\la}\int_0^t\nres{\tau}u(t-\tau)d\tau}}}
$$
and consequently $T_{\la}\psi_{|\rep} \in Z$ for every $\la>0.$
Assuming that $(u^+_{n-1,\la})_{|\rep}\in Z$ we have $(u_{n-1,\la})_{|\rep}\in Y$ for every $\la>0.$ Consequently, we obtain for every $\la >0$ that
$$
\Funk{F_{n,\la}}{Z}{Z}{u}{\bk{t\mapsto \Jl^{\om}(t,(u_{n-1,\la})_t)\fk{\nres{t}\vp(0)+\frac{1}{\la}\int_0^t\nres{\tau}u(t-\tau)d\tau}},}
$$
has fixpoint $u\in Z$ which gives 
$$
u_{n,\la}(t)=T_{\la,\vp}^n\psi(t)=
\left\{\begin{array}{rcl}
u(t)  &:& t\in\jz \\
\vp_{u_{n,\la},\la}(t) &:& t \in \ji,
\end{array} \right.
$$
which is an element of $Y,$ for all $n\in \za$ and $\la >0.$
Applying Theorem \ref{halfline-uniform-convergent} we obtain
$$
\lim_{n\to\infty}\lim_{\la\to 0}u_{n,\la}=u\in Y,
$$
 and the proof is finished.
\end{proof}

\section{Integral Solutions of Abstract Functional Differential Equations}
In this section we will show that the solution found in Theorem \ref{0-T-solution} is, under some prerequisites the mild or integral solution to the corresponding Cauchy Problem, with $x_0=\vp(0).$

\begin{equation} \label{Cauchy-equation}
\left\{
\begin{array}{rcll} 
w^{\prime}(t) &\in & B(t)w(t)+\om w(t)+f(t)&:t\in\jz\\
w(0) &=& x_0,
\end{array}
\right.
\end{equation}

with  $B(t)=A(t,u_t)$ and $u$ the solution to (\ref{gFDEonInterval}). In doing this we need some additional regularity of the solution $u.$
\begin{defi}
A function $u:[a,b]\to X$ is called a mild solution to (\ref{Cauchy-equation}) if there exist $u_{\la}:[a,b]\to X$
such that
\begin{eqnarray*}
(\partial_t)_{\la}u_{\la}(t)&\in& B(t)u_{\la}(t)+\om u_{\la}(t)+f(t) \mbox{ for all } t\in [a,b] \\
\ul(0)&\to& x_0 \mbox{ when } \la \to 0.
\end{eqnarray*}
and $\ul\to u $ on $[a,b]$ when $\la\to 0.$
\end{defi}

\begin{lem} \label{u-vp}

Let for $\jz=[0,T]$ Assumption \ref{IVFDE_1} and either Assumption \ref{IVFDE_2_og} or \ref{IVFDE_2} hold.
If $\vp$ Lipschitz, and $\vp(0)\in\hat{D},$ $t\in[0,T],$ then for some $K>0,$
\begin{enumerate}
\item
$$ \btr{A(t,(\up{\la}{n-1})_t)\vp(0)} \le K \mbox{ for all } \la>0, n\in\za$$
\item
$$
\nrm{\up{\la}{n}(t)-\vp(0)}\le K(\la+t), \mbox{ for all } \la>0, n\in\za.
$$
\end{enumerate}
\end{lem}
\begin{proof}  
As $\bk{\up{\la}{n}}_{\la>0,n\in\za}$ is uniformly bounded and $\vp(0)\in \hat{D}$ by Remark \ref{A_t_vp_uniformly_bd} we find some $K>0,$ such that
$$
\btr{A(t,(\up{\la}{n})_t)\vp(0)}\le K \mbox{ for all } \la>0, n\in\za.
$$
In order to prove the second inequality,
\begin{eqnarray*}
\lefteqn{\nrm{\up{\la}{n}(t)-\vp(0)}}\\
&\le&\nrm{\Jlw(t,(\up{\la}{n-1})_t)
\fk{\nres{t}\vp(0)+\frac{1}{\la}\int_0^t\nres{\tau}\up{\la}{n}(t-\tau)d\tau}-\vp(0)} \\
&\le& \nrm{\Jlw(t,(\up{\la}{n-1})_t)\fk{\nres{t}\vp(0)+\frac{1}{\la}\int_0^t\nres{\tau}\up{\la}{n}(t-\tau)d\tau}-\Jlw(t,((\up{\la}{n-1})_t)\vp(0)} \\
&&+\nrm{\Jlw(t,(\up{\la}{n-1})_t)\vp(0)-\vp(0)} \\
&\le& \frac{1}{\la(1-\la\om)}\int_0^t \nres{\tau}\nrm{\up{\la}{n}(t-\tau)-\vp(0)}d\tau 
+\nrm{\Jlw(t,(\up{\la}{n-1})_t)(t)\vp(0)-\vp(0)}\\
&\le& \frac{1}{\la(1-\la\om)}\int_0^t \nres{\tau}\nrm{\up{\la}{n}(t-\tau)-\vp(0)}d\tau\\
&& +\la \fk{\btr{A(t,(\up{\la}{n-1})_t)(t)\vp(0)}+\btr{\om}\nrm{\vp(0)}}\\
&\le& \frac{1}{\la(1-\la\om)}\int_0^t \nres{\tau}\nrm{\up{\la}{n}(t-\tau)-\vp(0)}d\tau +\la K
\end{eqnarray*}
Hence, by \cite[Lemma A.8., p. 1096]{Kreulichevo} we have for some adequate $K>0$
\begin{eqnarray*}
\nrm{\up{\la}{n}(t)-\vp(0)} &\le& \la K + \int_0^t K d\tau = K(\la+t).
\end{eqnarray*}
\end{proof}

\begin{lem} \label{first-inequalities1}
Let for $\jz=[0,T]$ Assumption \ref{IVFDE_1} and Assumption \ref{IVFDE_2} hold. Further, let $\vp$ Lipschitz, $\vp(0)\in\hat{D},$ $t\in[0,T],$ and $t>s>0,$ then for some $K^{\prime}>0,$ 
\begin{eqnarray*}
\lefteqn{\sup_{0<\al\le t}\limsup_{s\to 0}\frac{1}{s}\nrE{(\up{\la}{n})_{\al-s}-(\up{\la}{n})_{\al}}} \\
&& \le \sup_{0\le\al\le t}\limsup_{s\to 0}\frac{1}{s}\nrm{\up{\la}{n}(\al-s)-\up{\la}{n}(\al)}+(L_{\vp}+K^{\prime}).
\end{eqnarray*}
\end{lem}

\begin{proof}
As a consequence of Remark \ref{vp-la-minus-vp} and the definition of $u_{\la,n},$ we obtain 
$$
\nrE{(\up{\la}{n})_{t-s}-(\up{\la}{n})_t} \le \sup_{0<\al\le t}\nrm{\up{\la}{n}(\al-s)-\up{\la}{n}(\al)}+s(L_{\vp}+K^{\prime})
$$
which leads by Lemma \ref{limsup_sup_compare} to,

\begin{eqnarray*} 
\lefteqn{\limsup_{s\to 0+}\frac{1}{s}\sup_{r\in \ji}\nrm{u_{\la,n}(r +\alpha-s)-u_{\la,n}(r+\alpha)}}  \\
&\le &  \limsup_{s\to 0+}\sup_{0\le\al\le t}\frac{1}{s}\nrm{u_{\la,n}(\alpha-s)-u_{\la,n}(\alpha)} + L_{\vp}+K^{\prime}  \\
&\le & \sup_{0<\al\le t} \limsup_{s\to 0+}\frac{1}{s}\nrm{u_{\la,n}(\alpha-s)-u_{\la,n}(\alpha)} + L_{\vp}+K^{\prime} .
\end{eqnarray*}
Computing the $\sup$ on the left hand side completes the proof.
\end{proof}

From Lemma \ref{iteration-possible} c) we obtain that if the starting point $\psi\in Y,$ in the recursion, and $\vp\in E$ are Lipschitz, and $\vp(0) \in \hat{D}$ then $u_{\la,n}$ is Lipschitz for every $\la>0,$ and $n\in\za_0$. In Lemma \ref{equi_Lipschitz} we will show that they are equi-Lipschitz.

\begin{lem} \label{equi_Lipschitz}
Let for $\jz=[0,T]$ Assumption \ref{IVFDE_1} and Assumption \ref{IVFDE_2} hold. If  the starting point of the recursion $\psi\in Y,$ and $\vp\in E$ are Lipschitz, and $\vp(0)\in \hat{D},$ then the set of functions $\bk{\up{\la}{n}:\la >0, n\in \za}$ is equi-Lipschitz. Hence the Yosida approximations
$$
(\partial_t)_{\la}\up{\la}{n}(t):=\frac{1}{\la}\fk{\up{\la}{n}(t)-\vp(0)-\frac{1}{\la}\int_0^t\nres{\tau}(\up{\la}{n}(t-\tau)-\vp(0))d\tau}
$$
are uniformly bounded for small $\la>0$ and $n\in \za.$
\end{lem}

\begin{proof}
By Assumption \ref{IVFDE_2} and Recursion \ref{recursion} we have,
\begin{eqnarray*} 
\lefteqn{\nrm{u_{\la,1}(t-s)-u_{\la,1}(t)}} \\
&\le& \frac{1}{1-\la\om}\nrm{u_{\la,1}(t-s)-u_{\la,1}(t)-\la\bk{(\partial_t)_{\la}u_{\la,1}(t-s)-(\partial_t)_{\la}u_{\la,1}(t)}} \\
&&+\frac{\la}{1-\la\om} \nrm{h^{\om}(t-s)-h^{\om}(t)}L^{\om}_1(\nrm{u_{\la,1}(t)}) \\
&&+\frac{\la}{1-\la\om}\nrm{g(t-s)-g(t)} \nrm{\frac{1}{\la}\fk{u_{\la,1}(t)-\vp(0)-\int_0^t\nresl{\tau}(u_{\la,1}(t-\tau)-u_0)d\tau}}\\
&& +\frac{\la}{1-\la\om} \nrm{k(t-s)-k(t)}L_2(\nrE{\psi_t}) \\
&&+\frac{K_0}{1-\la\om}\nrE{\psi_{t-s}-\psi_t} \\
&\le& \frac{1}{1-\la\om} 
\nrm{\nres{t-s}\vp(0)+\int_0^{t-s}\nresl{\tau}u_{\la,1}(t-s-\tau)d\tau
   - \nres{t}\vp(0)-  \int_0^{t}  \nresl{\tau}u_{\la,1}(t-\tau)d\tau}  \\
&&+\frac{\la}{1-\la\om} \nrm{h^{\om}(t-s)-h^{\om}(t)}L^{\om}_1(\nrm{u_{\la,1}(t)}) \\
&&+\frac{\la}{1-\la\om}\nrm{g(t-s)-g(t)} \nrm{\frac{1}{\la}\fk{u_{\la,1}(t)-\vp(0)-\int_0^t\nresl{\tau}(u_{\la,1}(t-\tau)-\vp(0))d\tau}},\\
&& +\frac{\la}{1-\la\om} \nrm{k(t-s)-k(t)}L_2(\nrE{\psi_t}) \\
&&+\frac{K_0}{1-\la\om}\nrE{\psi_{t-s}-\psi_t} \\
&\le&\frac{1}{\la(1-\la\om)}\int_0^{t-s}\nres{\tau}\nrm{u_{\la,1}(t-s-\tau)-u_{\la,1}(t-\tau)}d\tau \\
&&+\frac{1}{\la (1-\la\om)}\int_{t-s}^t\nres{\tau}\nrm{u_{\la,1}(t-\tau)-\vp(0)}d\tau  \\
&&+\frac{\la}{1-\la\om} \nrm{h^{\om}(t-s)-h^{\om}(t)}L^{\om}_1(\nrm{u_{\la,1}(t)})\\
&&+\frac{\la}{1-\la\om}\nrm{g(t-s)-g(t)} \nrm{\frac{1}{\la}\fk{u_{\la,1}(t)-\vp(0)-\int_0^t\nresl{\tau}(u_{\la,1}(t-\tau)-\vp(0))d\tau}}\\
&& +\frac{\la}{1-\la\om} \nrm{k(t-s)-k(t)}L_2(\nrE{\psi_t}) \\
&&+\frac{K_0}{1-\la\om}\nrE{\psi_{t-s}-\psi_t}. 
\end{eqnarray*}

Defining for $s>0,$ and $n\in\za_0,$
\begin{eqnarray*}
K_{\la,n}(t)&:=&\limsup_{s\to 0+}\frac{1}{s}\nrm{u_{\la,n}(t-s)-u_{\la,n}(t)} \\ K_{\la,n}^{\sup}(t)&:=& \sup_{0\le\alpha\le t}K_{\la,n}(\alpha) \\
S_{\la,n}(t)&:=&\sup_{0\le\al\le t}\limsup_{s\to 0+}\frac{1}{s}\nrE{(u_{\la,n})_{\alpha-s}-(u_{\la,n})_{\alpha}}. \\
\end{eqnarray*}
Note that by Lemma \ref{first-inequalities1}  we have,
\begin{eqnarray}\label{S_la-K_la-compare}
S_{\la,n}(t)&\le& K_{\la,n}^{\sup}(t) + L_{\vp}+K^{\prime}.
\end{eqnarray}

With the use of the notations above we estimate the Yosida approximation of the  derivative,
\begin{eqnarray*}
\frac{1}{\la}\nrm{u_{\la,1}(t)-u_0-\frac{1}{\la}\int_0^t\nres{\tau}(u_{\la,1}(t-\tau)-u_0)d\tau}
&\le& \frac{1}{\la^2}\int_0^t\nres{\tau}\nrm{u_{\la,1}(t-\tau)-u_{\la,1}(t)}d\tau \\
&&+ \frac{1}{\la}\nres{t}\nrm{u_{\la,1}(t)-u_0}.
\end{eqnarray*}
The second term is quite simple to estimate, as

\begin{eqnarray}
\frac{1}{\la}\nres{t}\nrm{u_{\la,1}(t)-u_0}
&\le& \frac{1}{\la}\nres{t}K^{\prime}(\la+t ) ,\\
&\le&  K^{\prime}\nres{t} +K^{\prime}\frac{t}{\la}\nres{t}, \\
&\le& K_1.
\end{eqnarray}
Applying Proposition \ref{bd-limsup-Lipschitz} for given $t>0,$ we have
$$
\sup_{0<s<t}\frac{1}{s}\nrm{u_{\la,1}(t-s)-u_{\la,1}(t)} \le K_{\la,1}^{\sup}(t).
$$

Hence, for the integral we obtain
\begin{equation}
\frac{1}{\la^2}\int_0^t\nres{\tau}\frac{\tau}{\tau}\nrm{u_{\la,1}(t-\tau)-u_{\la,1}(t)}d\tau \le \frac{1}{\la^2}\int_0^{t}\nres{\tau}\tau d\tau K_{\la,1}^{\sup}(t) \le  K^{\sup}_{\la,1}(t), 
\end{equation}

and conclude
\begin{eqnarray*}
\frac{1}{\la}\nrm{u_{\la,1}(t)-u_0-\frac{1}{\la}\int_0^t\nres{\tau}(u_{\la,1}(t-\tau)-u_0)d\tau}
&\le& K^{\sup}_{\la,1}(t)+K_1.
\end{eqnarray*}

Applying Lemma \ref{u-vp}
$$
\nrm{\up{\la}{n}(t)-\vp(0)}\le K(\la+t), \mbox{ for all } \la>0, n\in\za_0,
$$
we have 
\begin{eqnarray*}
\lefteqn{\frac{1}{s\la}\int_{t-s}^t\nres{\tau}\nrm{u_{1,\la}(t-\tau)-\vp(0)}d\tau} \\
&\le& \frac{K}{\la s} \int_{t-s}^t\nres{\tau}(\la+(t-\tau))d\tau =  \frac{K}{\la s} \int_0^s\nres{t-\tau}(\la+\tau)d\tau \\
&=& K\frac{\nres{t}}{s\la}\fk{\la\int_0^se^{\frac{\tau}{\la}}d\tau
+\la se^{\frac{s}{\la}} -\la\int_0^se^{\frac{\tau}{\la}}d\tau} =K\nres{t-s}.
\end{eqnarray*}

The boundedness of $\bk{u_{n,\la}}_{\la>0,n\in \za_0}$ leads to a $C_u, $ such that $L_1(\nrm{u_{n,\la}(t)},L_2(\nrE{(u_{n,\la})_t})\le C_u$ for all $0\le t\le T$, $ n\in\za_0, $ and $ \la >0.$ Moreover , let $\la\om \le q < 1$ and choose $C_u$ such that $\frac{K_0}{1-\la\om}\le C_u, $ witch implies for the inequality,
\begin{eqnarray*}
\lefteqn{(1-\la\om)\frac{1}{s}\nrm{u_{1,\la}(t-s)-u_{1,\la}(t)}} \\
&\le& \int_0^{t-s}e^{-\frac{t}{\la}}\frac{1}{s}\nrm{u_{1,\la}(t-s-\tau)-u_{1,\la}(t-\tau)}d\tau + K\exp\fk{\frac{s-t}{\la}} \\
&& + \la C_u \fk{L_{h^{\om}}+L_k+S_{\la,0}(t)}  + \la L_g K_1 +\la L_g K_{\la,1}^{\sup}(t).
\end{eqnarray*}

If passing to $\limsup_{s\to 0}$ on both sides of the inequality, we have with some adequate $C_2,$
\begin{eqnarray*}
(1-\la\om)K_{\la,1}(t) &\le& \frac{1}{\la}\int_0^{t}e^{-\frac{t}{\la}}K_{ \la,1}(t-\tau)d\tau + \la\fk{C_2+ C_uS_{\la,0}(t)} \\
&& +\la L_g K_{\la,1}^{\sup}(t) + K\exp\fk{\frac{-t}{\la}}.
\end{eqnarray*}
Hence we are in the situation to apply Proposition \ref{s-t-integral-inequality}, and obtain
\begin{eqnarray*}
K_{\la,1}(t)
 &\le& \la L_g K_{\la,1}^{\sup}(t) + K\exp\fk{\frac{-t}{\la}} + \la\fk{C_2+ C_uS_{\la,0}(t)}\\
 &&+\frac{1}{\la(1-\la\om)} \int_0^t\exp\fk{\frac{\om(t-\tau)}{1-\la\om}}
 \fk{ K\exp\fk{\frac{-\tau}{\la}} + \la C_2}d\tau \\
&&+\frac{1}{\la(1-\la\om)} \int_0^t\exp\fk{\frac{\om(t-\tau)}{1-\la\om}}
\la\fk{ L_g K_{\la,1}^{\sup}(\tau) +  C_uS_{\la,0}(\tau)}d\tau
\end{eqnarray*}
Similar to \cite[p. 1067]{Kreulichevo} we have,
$$
\frac{1}{\la(1-\la\om)} \int_0^t\exp\fk{\frac{\om(t-\tau)}{1-\la\om}}K\exp\fk{\frac{-\tau}{\la}}d\tau \le C_3.
$$
Thies yields for some adequate $C_4$ we conclude
\begin{eqnarray*}
\lefteqn{K_{\la,1}(t) \le \frac{1}{1-\la\om}\fk{C_4 +\la L_g K_{\la,1}^{\sup}(t)+\la C_uS_{\la,0}(t)}} \\
&&+\frac{C_u}{(1-\la\om)} \int_0^t\exp\fk{\frac{\om(t-\tau)}{1-\la\om}}
S_{\la,0}(\tau) d\tau +\frac{L_g}{(1-\la\om)} \int_0^t\exp\fk{\frac{\om(t-\tau)}{1-\la\om}} K_{\la,1}^{\sup}(\tau) d\tau
\end{eqnarray*} 
As the right hand side is monotone increasing we find,
\begin{eqnarray*}
\lefteqn{K_{\la,1}^{\sup}(t) \le \frac{1}{1-\la\om}\fk{C_4+ \la L_g K_{\la,1}^{\sup}(t)+\la C_uS_{\la,0}(t)} } \\
&&+\frac{C_u}{(1-\la\om)} \int_0^t\exp\fk{\frac{\om(t-\tau)}{1-\la\om}}
S_{\la,0}(\tau) d\tau+\frac{L_g}{(1-\la\om)} \int_0^t\exp\fk{\frac{\om(t-\tau)}{1-\la\om}} K_{\la,1}^{\sup}(\tau) d\tau.
\end{eqnarray*}
A second application of Proposition \ref{s-t-integral-inequality} leads with
$$
\gamma_{\la}=\frac{\om+L_g-\la \om (\om+2L_g)}{(1-\la\om)(1-\la(\om+L_g))},
$$
to
\begin{eqnarray*}
\lefteqn{K_{\la,1}^{\sup}(t)
 \le \frac{1}{1-\la(\om+L_g)}\fk{C_4 +\la C_u S_{\la,0}(t) +C_u \int_0^t\exp\fk{\frac{\om(t-\tau)}{1-\la\om}} S_{\la,0}(\tau) d\tau } }\\
&&+  \fk{\frac{L_g}{1-\la(\om+L_g)}}^2\int_0^t\exp(\gamma_{\la}(t-\tau)
\fk{C_4 +\la C_u S_{\la,0}(\tau)} d\tau \\
&&+\fk{\frac{L_g}{1-\la(\om+L_g)}}^2\int_0^t\exp(\gamma_{\la}(t-\nu))
 \int_0^\nu\exp\fk{\frac{\om(\nu-\mu)}{1-\la\om}} C_u S_{\la,0}(\mu) d\mu d\nu.
\end{eqnarray*} 
As we are on a bounded interval $0\le t\le T$ for some adequate constant $ C_5,$ we obtain

\begin{eqnarray*}
K_{\la,1}^{\sup}(t)
&\le & C_5 \fk{ 1+ \la S_{\la,0}(t)} + C_5(1+\la) \int_0^t S_{\la,0}(\mu) d\mu + C_5 \int_0^t \int_0^{\nu} S_{\la,0}(\mu) d\mu d\nu.
\end{eqnarray*} 
Using the boundedness of $\bk{\vp_{u_{n,\la},\la}}_{n\in\za, \la>0},$ compare Remark \ref{vp-la-minus-vp}, the monotonicity of the right hand side, (\ref{S_la-K_la-compare}) and defining the integral operator
$$
\Funk{T}{C[0,T]}{C[0,T]}{f}{\bk{t\mapsto \int_0^t f(\tau)d\tau},}
$$
the integral inequality becomes with some adequate $C_6$
\begin{eqnarray*}
S_{\la,1}(t)&\le & C_6\fk{1+ \la S_{\la,0}(t) + (1+\la)TS_{\la,0}(t)+ T^2S_{\la,0}(t)}.
\end{eqnarray*}

As $\psi$  is arbitrary in $Y$ and Lipschitz and the inequalities neither depend on $\la>0$ nor on $n\in\za,$  we can do the induction step with the same methods. Hence,
\begin{eqnarray*}
S_{\la,n+1}(t)&\le & C_6\fk{1+ \la S_{\la,n}(t) + (1+\la)TS_{\la,n}(t)+ T^2S_{\la,n}(t)}.
\end{eqnarray*}
Due to $ T$ quasi-nilpotent by the Spectral Mapping Theorem \cite[Thm. 10.28]{RudinFA}  $\sigma(\alpha T+\beta T^2)=\bk{0},$ for all $\alpha,\beta\in\re.$ Hence for $ \la C_6 \le q < 1  $ we are in the situation of Proposition \ref{quasinilpotent-boundedness}, and obtain the uniform bound for the Lipschitz constants of the family $\bk{u_{\la,n}}_{n\in\za,\la<q/C_6}.$

\end{proof}

In order to compare the Cauchy problem coming with $B(t)=A(t,u_t)$ we recall the Assumptions in the case of a non-autonomous Cauchy problem discussed in the study \cite{Kreulichevo}.

\begin{assu}\label{general-IVA0} 
  The set $\bk{B(t):t \in \jz}$ is a family of m-dissipative operators
\end{assu}
\begin{assu}\label{general-IVE1}
  There exist  $h\in BUC(\jz,X),  $ and $L_1: \rep \lra \rep$, continuous and monotone non-decreasing,
such that for $\la >0,$ and $t_1,t_2\in \jz$ we have
\begin{eqnarray*}
\lefteqn{\nrm{x_1-x_2}} \\
&\le& \nrm{x_1-x_2 -\la(y_1-y_2)}+\la\nrm{h(t_1)-h(t_2)}L_1(\nrm{x_2}),
\end{eqnarray*}
for all $[x_i,y_i]\in B(t_i),$  $i=1,2.$
\end{assu}
\begin{assu}\label{general-IVE2}
There are bounded and Lipschitz continuous functions $g,h:\jz \to X, $ 
and $L_1: \rep \lra \rep$ continuous, and monotone non-decreasing, 
such that for $\la >0,$ and $t_1,t_2\in \jz,$ we have
\begin{eqnarray*}
\lefteqn{\nrm{x_1-x_2}} \\
&\le& \nrm{x_1-x_2 -\la(y_1-y_2)}+\la \nrm{h(t_1)-h(t_2)}L_1(\nrm{x_2}) \\
&&+ \la \nrm{g(t_1)-g(t_2)}\nrm{y_2},
\end{eqnarray*}
for all $[x_i,y_i]\in B(t_i),$  $i=1,2.$
\end{assu}

\begin{theo} \label{general_mild_solution_lipschitz}
Let for $\jz=[0,T]$ Assumption \ref{IVFDE_1} and Assumption \ref{IVFDE_2} hold. Further, let  $\vp\in E$ Lipschitz, $\vp(0)\in \hat{D},$ $u$ the solution to 
(\ref{gFDEonInterval}), $ B(t)=A(t,u_t),$ and $f=0,$  then the solution $u\in Y$ found in Theorem \ref{0-T-solution} is a mild solution to
(\ref{Cauchy-equation}).

\end{theo}

\begin{proof}
From Lemma \ref{equi_Lipschitz} we learn that $u_t$ is the uniform limit of the equi-Lipschitz family $\bk{u_{\la,n}}_{n\in\za_0,\la >0},$ consequently
$$\nrE{u_t-u_s}\le L\btr{t-s}$$ Thus, we obtain with the modified $h$
\begin{equation} \label{ut-modified-control-1}
\Funk{\tilde{h}}{\jz}{(X\times X\times E, \nrm{\cdot}_1)}{t}{(h^{\om}(t),k(t),u_t),}
\end{equation}
the modified $L,$
\begin{equation} \label{ut-modified-control-2}
L_B(t)=L_1(t)+L_2(\sup_{s\in[0,T]}\nrE{u_s})+\frac{K_0}{1-\la\om},
\end{equation}
and by Assumption \ref{IVFDE_2} that $B(t)=A(t,u_t)$ satisfies Assumption \ref{general-IVE2} with the previously defined control functions. 

Thus, by \cite[Theorem 2.9]{Kreulichevo} the approximation
\begin{equation} \label{approx-Cauchy-equation}
\left\{
\begin{array}{rcll} 
\fk{\partial_t}_{\la} w_{\la}(t) &\in & B(t)w_{\la}(t)+\om w_{\la}(t)&:t\in\jz\\
w(0) &=& \vp(0).
\end{array}
\right.
\end{equation}
tends to the integral solution of (\ref{Cauchy-equation}). On the other hand we have the  approximation $u_{\la,n}$ of the generalized FDE given by Recursion \ref{recursion}, with 
$$
\lim_{n\to\infty}\lim_{\la\to 0}u_{\la,n}=u.
$$
 Thus we are in the situation of Lemma  \ref{induction_step_1}, which concludes the proof.

\end{proof}

\begin{theo} \label{general_mild_solution_cont}
Let for $\jz=[0,T]$ Assumption \ref{IVFDE_1} and Assumption \ref{IVFDE_2_og} hold. Further, let  $\vp\in E$ Lipschitz, $\vp(0)\in \hat{D},$ $u$ the solution to 
(\ref{gFDEonInterval}) and $ B(t)=A(t,u_t),$  then the solution $u\in Y$ found in Theorem \ref{0-T-solution} is a mild solution to
(\ref{Cauchy-equation})
\end{theo}

\begin{proof}
As $\bk{t\mapsto u_t}$ is continuous we choose control functions for $B(t):=A(t,u_t)$ similar to (\ref{ut-modified-control-1}) and (\ref{ut-modified-control-2}), and we are in the situation of the proof of Theorem \ref{general_mild_solution_lipschitz}, with Assumption \ref{general-IVE1} instead of Assumption \ref{general-IVE2}.
\end{proof}

\begin{defi} \label{def-int-sol-inhomo}
Assume that either the Assumption \ref{general-IVE1} or Assumption \ref{general-IVE2} is satisfied for the family $\bk{B(t):t\in \jz}.$ In the case of Assumption \ref{general-IVE1} choose $g=0.$ Let $\jz=\rep.$ A continuous function 
$u:[a,b]\to X$ is called an integral solution of (\ref{Cauchy-equation}) if $u(a)=x_0$ and
\begin{eqnarray*} 
\lefteqn{\nrm{u(t)-x}-\nrm{u(r)-x}} \\
&\le &\int_r^t\fk{[y+f(\nu),u(\nu)-x]_{+}+\om\nrm{u(\nu)-x}}d\nu \\
&&+\Lw(\nrm{x})\int_r^t\nrm{\hw(\nu)-\hw(r)}d\nu +\nrm{y}\int_r^t\nrm{g(\nu)-g(r)}d\nu
\end{eqnarray*}
for all $a\le r\le t\le b,$ and $s\in [a,b],$ $[x,y]\in B(s)+\om I.$
\end{defi}

To view the found solution of the functional differential equation as an integral we  have to slightly weaken the assumption on $\hw$ in the case of Assumption \ref{IVFDE_2}. The $\hw$ is only continuous, but in view of regularity compared with the proof of \cite[Theorem 6.37]{Ito_Kappel} it is still a sufficient condition.

\begin{theo} \label{general_integral-solution-lipschitz}
Let for $\jz=[0,T]$ Assumption \ref{IVFDE_1} and Assumption \ref{IVFDE_2} hold. If $\vp\in E$  with $\vp(0)\in\overline{D},$ 
 then the solution $u\in Y$ found in Theorem \ref{0-T-solution} is an integral solution to
 (\ref{Cauchy-equation}) with $B(t)=A(t,u_t)$ and $f=0.$ That is $u$ satisfies Definition \ref{def-int-sol-inhomo} with an only continuous $\hw,$ and adequate $\Lw,$ and $g.$
\end{theo}

\begin{proof} 
To prove the claim we want to apply Theorem \ref{general_mild_solution_lipschitz}, which needs a Lipschitz initial history with $\vp(0)\in \hat{D}$ In doing so we will construct an appropriate approximation. For this purpose let the starting point $\psi$ and $\tilde{\vp}^m\in E$ Lipschitz with $\nrE{\tilde{\vp}^m-\vp}\le \frac{1}{m}.$  As $\overline{\hat{D}}=\overline{D}$ we find $\bk{x_m}_{m\in\za}\subset \hat{D}$ such that $\nrm{x_m- \vp(0)}\le \frac{1}{m}. $
Defining  
$$
\vp^m(t):=\left\{\begin{array}{rcl} 
					-mt\tilde{\vp}^m(1/m)+(1+mt)x_m &:& -\frac{1}{m}\le t \le 0 \\
					\tilde{\vp}^m(t) &:& t<-\frac{1}{m}
					\end{array} \right.
$$
we prove $\vp^m\to\vp$ 
We only have to verify the convergence for $-\frac{1}{m}\le t \le 0. $  For such $t$ we have
\begin{eqnarray*}
\lefteqn{\nrm{\vp^m(t)-\vp(t)}} \\
&\le& \nrm{-mt(\tilde{\vp}(1/m)-\vp(t))} +\nrm{(1+mt)(x_m-\vp(0)}+\nrm{(1+mt)(\vp(0)-\vp(t))} \\
&\le&\nrm{\tilde{\vp}(1/m)-\vp(1/m}+\nrm{\vp(1/m)-\vp(t)}+ \nrm{(x_m-\vp(0)}+ \nrm{\vp(0)-\vp(t)} \\
&& \to 0 \mbox{ when } m\to \infty,
\end{eqnarray*}

by the uniform continuity of $\vp$ and $\nrm{x_m-\vp(0)}\le \frac{1}{m},$
Next we claim $\vp^m$ is Lipschitz with $L_{\vp^m}=L_{\tilde{\vp}^m}+2.$ As $\tilde{\vp}^m$ is Lipschitz, it remains to consider $-\frac{1}{m}\le t,s \le 0$
\begin{eqnarray*}
\lefteqn{\nrm{-tm \tilde{\vp}^m(1/m)+(1+mt)x_m +sm \tilde{\vp}^m(1/m)-(1+ms)x_m}} \\
&\le& \nrm{(s-t)m (\tilde{\vp}^m(1/m) -\tilde{\vp}^m(0))} +\nrm{(s-t)m(x_m-\tilde{\vp}^m(0))} \\
&\le& \btr{s-t}\fk{L_{\tilde{\vp}^m}+m\nrm{x_m-\tilde{\vp}^m(0)}}\\
&\le& \btr{s-t}\fk{L_{\tilde{\vp}^m} +2}.
\end{eqnarray*}

With these settings we consider the following approximation of the functional differential equation,

\begin{equation}  \nonumber
\left\{
\begin{array}{rcl} 
(\partial_t)_{\la}u^m_{\la,n}(t) &\in & A(t,(u^m_{\la,n-1})_t)u^m_{\la,n}(t)+ \om u^{m}_{\la,_n(}t): \ t \in \jz \\
(u^m_{\la,n})_{|\ji}&=&\vp^m_{u^m_{\la,n-1},\la}.
\end{array}
\right.
\end{equation}
Note that
$$
u^m_{\la,n}=T_{\la,\vp^m}^n\psi.
$$

As $\vp^m$ is Lipschitz and $\vp^m(0)\in\hat{D}$ we find by Lemma (\ref{equi_Lipschitz}) $u^m_n=\lim_{\la\to 0}u_{\la,n}^m$ is Lipschitz. Moreover $B(t)=A(t,(u^m_n)_t)$ satisfies Assumption \ref{general-IVE2} with a modified "$h$" defined by, 

\begin{equation} \nonumber
\Funk{h_n^m}{\jz}{(X\times X\times E, \nrm{\cdot}_1)}{t}{(h^{\om}(t),k(t),(u_{n-1}^m)_t),}
\end{equation}
and
\begin{equation} \
L_B^m(t)=L_1(t)+L_2\fk{\sup_{n\in\za,s\in[0,T]}\nrE{(u_n^m)_s}}+\frac{K_0}{1-\la\om}.
\end{equation}
Let $h^m:=\lim_{n\to\infty}h_n^m.$

If the initial value of the Cauchy Problem equals $\vp^m(0),$ by \cite[Theorem 2.9]{Kreulichevo} the integral solution comes with the limit $\la\to 0$ of

\begin{equation} \nonumber
\left\{
\begin{array}{rcl} 
(\partial_t)_{\la}v^m_{\la,n}(t) &\in & A(t,(u^m_{n-1})_t)v^m_{\la,n}(t)+ \om v^{m}_{\la,n}(t): \ t \in \jz \\
v^m_{\la,n}(0)&\to& \vp^m(0) \mbox{ when } \la \to 0.
\end{array}
\right.
\end{equation}

Consequently, from the recursion we have  $\lim_{\la\to 0}u^m_{\la,n}=u^m_n$ and Lemma \ref{induction_step_1} leads to $v^m_n=u^m_n.$ Hence, $u^m_n$ is an integral solution and we find for every $n,m\in \za,$
\begin{eqnarray} \label{approx-integral-solution-ineq} 
\lefteqn{\nrm{u^m_n(t)-x^m_n}-\nrm{u^m_n(r)-x^m_n}} \nonumber \\
&\le &\int_r^t\fk{[y^m_n,u^m_n(\nu)-x^m_n]_{+}+\om\nrm{u^m_n(\nu)-x^m_n}}d\nu \\
&&+L_B^m(\nrm{x^m_n})\int_r^t\nrm{h_n^m(\nu)-h_n^m(r)}d\nu +\nrm{y^m_n}\int_r^t\nrm{g(\nu)-g(r)}d\nu \nonumber
\end{eqnarray}
for all $a\le r\le t\le b,$ and $s\in [a,b],$ $[x^m_n,y^m_n]\in A(s,(u^m_{n})_s)+\om I.$ Applying Remark \ref{history_control}, we have
$$
\nrm{J_{\la}(t,u^m_n)_s)x -J_{\la}(t,u^m)x}\le\frac{\la K_0}{1-\la\om}\nrm{u^m_n-u^m}_E,
$$
and the convergence of the recursion gives $\lim_{n\to\infty}u^m_n=u^m.$ Using Assumption \ref{IVFDE_1} we are in the situation of \cite[Theorem 10.5.]{Ito_Kappel} and \cite[(10.6)]{Ito_Kappel}. Thus, for $[x^m,y^m]\in A(t,u^m_t)$ we find
$[x^m_n,y^m_n]\in A(t,(u^m_n)_t)$ such that $\lim_{n\to\infty}(x^m_n,y^m_n)=(x^m,y^m).$ Hence we may pass to the limit $n\to\infty$ in the inequality (\ref{approx-integral-solution-ineq}).

It remains to do the limit $m\to \infty,$ which will be obtained with a stability result on the initial history. From Lemma \ref{psi-la-approx-gFDE} inequality (\ref{iterate_ineq}) we obtain
\begin{eqnarray*}
\nrm{u^l_{\la,n}(t) -u^k_{\la,n}(t)}
&\le& \frac{\la K_0}{1-\la\om}\nrm{(u^l_{\la,n-1})_t-(u^k_{\la,n-1})_t}_E  \\
&&+\frac{K_0}{1-\la\om}\int_0^t\exp\fk{\frac{\om}{1-\la\om}\tau}\nrm{(u^l_{\la,n-1})_{t-\tau}-(u^k_{\la,n-1})_{t-\tau}}_E d\tau \\
&& +\fk{\frac{1}{1-\la\om}\exp\fk{-\frac{t}{\la}} +\exp\fk{\frac{\om t }{1-\la\om}}}\nrm{\vp^l(0)-\vp^k(0)}. 
\end{eqnarray*}
When passing to $\la\to 0 $ and $n\to \infty ,$ the estimate $\exp(\om t)\le \exp\fk{\btr{\om} T},$  and $u^l_{|\ji}=\vp^l,$ we obtain

\begin{eqnarray*}
\nrE{u^l_t -u^k_t}
&\le &K_0\exp\fk{\btr{\om} T} \int_0^t\nrm{u^l_{\tau}-u^k_{\tau}}_E d\tau +\exp\fk{\btr{\om} T}\nrE{\vp^l-\vp^k}. 
\end{eqnarray*}
An application of Proposition \ref{s-t-integral-inequality}  gives for an adequate constant $K_2$
\begin{eqnarray*}
\nrE{u^l_t -u^k_t}
&\le & K_2\nrE{\vp^l-\vp^k}.
\end{eqnarray*}
Hence, $\bk{u^l}_{l\in\za}$ is Cauchy in $Y$. Let $u:=\lim_{m\to\infty}u^m, $ $h:=\lim_{m\to\infty} h^m,$ and $L_B(t):=\sup_{m\in\za}L_B^m(t).$ Note that $h$ is only continuous as mentioned.
By a second application of  Remark \ref{history_control}, 
$$
\nrm{J_{\la}(t,u^m_s)x -J_{\la}(t,u_s)x}\le\frac{\la K_0}{1-\la\om}\nrm{u^m-u}_E,
$$ 
and Assumption \ref{IVFDE_1} together with \cite[Theorem 10.5.]{Ito_Kappel} and \cite[(10.6)]{Ito_Kappel} lead, for given $[x,y]\in A(t,u_t)$ to $[x^m,y^m]\in A(t,u^m_t)$  such that $\lim_{m\to\infty}[x^m,y^m]=[x,y].$
Recalling inequality (\ref{approx-integral-solution-ineq}) we obtain when passing to $m\to \infty,$
\begin{eqnarray*} 
\lefteqn{\nrm{u(t)-x}-\nrm{u(r)-x} } \\
&\le &\int_r^t\fk{[y,u(\nu)-x]_{+}+\om\nrm{u(\nu)-x}}d\nu \\
&&+L_B(\nrm{x})\int_r^t\nrm{h(\nu)-h(r)}d\nu +\nrm{y}\int_r^t\nrm{g(\nu)-g(r)}d\nu
\end{eqnarray*}
for all $a\le r\le t\le b,$ and $s\in [a,b],$ $[x,y]\in A(s,(u)_s)+\om I,$ which concludes the proof.
\end {proof}

\begin{cor}  \label{general_integral-solution-continuous}
Let for $\jz=[0,T]$ Assumption \ref{IVFDE_1} and Assumption \ref{IVFDE_2_og} hold. If $\vp\in E$  with $\vp(0)\in\overline{D},$ 
 then the solution $u\in Y$ found in Theorem \ref{0-T-solution} is an integral solution to
 (\ref{Cauchy-equation}) with $B(t)=A(t,u_t),$ and $g=0.$
\end{cor}

\section{Special Case $A(t,\vp)=B(t)+F(t,\vp)$}

In the following we discuss the special case of $A(t,\vp)=B(t)+F(t,\vp)$ and assume throughout this section that $F$ satisfies the Assumption \ref{control-F(t,vp)}. We will show that in this case the solution found in Theorem \ref{0-T-solution} coincides with the integral solution.

\begin{assu} \label{control-F(t,vp)}
Let $F:\rep\times E \to X$ such that for some $K_0>0,$ a non-decreasing $L_2:\rep\to\rep$ and $k:\rep\to X,$
\begin{equation}
\nrE{F(t,\vp_1)-F(s,\vp_2)}\le K_0\nrE{\vp_1-\vp_2}+\nrm{k(t)-k(s)}L_2(\nrm{\vp_2}_E),
\end{equation}
$\vp_i\in E,$  for $i=1,2.$
\end{assu}

\begin{remk} \label{F(t,vp)-satisfies-control}
By the triangle inequality we obtain that if $F$ satisfies Assumption \ref{control-F(t,vp)} with $k$ uniformly continuous and $B(t)$ satisfies Assumption \ref{general-IVE1}, then $A(t,\vp):=B(t)+F(t,\vp)$ satisfies Assumption \ref{IVFDE_2_og}, and if $F$ satisfies Assumption \ref{control-F(t,vp)} with $k$ Lipschitz and $B(t)$ satisfies Assumption \ref{general-IVE2}, then $A(t,\vp):=B(t)+F(t,\vp)$ satisfies Assumption \ref{IVFDE_2} 
\end{remk} 

To obtain an integral solution to the Functional-Differential Equation,
\begin{equation} \label{FDEonInterval}
\left\{
\begin{array}{rcl} 
u^{\prime}(t) &\in & B(t)u(t)+ \om u(t)+F(t,u_t): \ t \in \jz \\
u_{|\ji}&=&\vp
\end{array}
\right.
\end{equation}
we view $F(t,u_t)$ as the inhomogenity $f.$ So the definition becomes:

\begin{defi} \label{def-int-sol-FDE}
Assume that $F$ satisfies Assumption \ref{control-F(t,vp)}, and either the Assumption \ref{general-IVE1} or Assumption \ref{general-IVE2} is satisfied for the family $\bk{B(t):t\in \jz}.$ In the case of Assumption \ref{general-IVE1} choose $g=0.$ Let $\jz=\rep.$ A continuous function 
$u:[a,b]\to X$ is called an integral solution of (\ref{FDEonInterval}) if $u_{|\ji}=\vp$ and
\begin{eqnarray*} 
\lefteqn{\nrm{u(t)-x}-\nrm{u(r)-x}} \\
&\le &\int_r^t\fk{[y+F(\nu,u_{\nu}),u(\nu)-x]_{+}+\om\nrm{u(\nu)-x}}d\nu \\
&&+\Lw(\nrm{x})\int_r^t\nrm{\hw(\nu)-\hw(r)}d\nu +\nrm{y}\int_r^t\nrm{g(\nu)-g(r)}d\nu
\end{eqnarray*}
for all $a\le r\le t\le b,$ and $s\in [a,b],$ $[x,y]\in B(s)+\om I.$
\end{defi}

\begin{cor}
Let $\jz=[0,T],$ $\vp\in E$ with $\vp(0)\in \overline{D},$ $F$ satisfies Assumption \ref{control-F(t,vp)}, and $B(t)$ Assumption \ref{general-IVA0}. If $B(t)$ satisfies either Assumption \ref{general-IVE1} and $k$ is uniformly continuous, or Assumption \ref{general-IVE2} and $k$ is Lipschitz, then the solution $u$ of (\ref{FDEonInterval}) is an integral solution of (\ref{Cauchy-equation}), with $f(t)=F(t,u_t).$ 
\end{cor}

\begin{proof}
Let $\bk{u_{n,\la}}_{n\in\za,\la>0}$ the approximation given by Recursion \ref{recursion}, and $u_n=\lim_{\la\to 0}u_{n,\la}.$ Considering the Cauchy problem,
\begin{equation} 
\left\{
\begin{array}{rcl}
(\partial_{\la}v_{n+1,\la}(t)(t) &\in & B(t)v_{n+1,\la}(t)+ \om v_{n+1,\la}(t)+F(t,u_{n,\la}(t)): \ t \in \jz \\
v_{n+1,\la}(0)&=&\vp_{v_{n,\la},\la}(0)
\end{array}
\right.
\end{equation}
By \cite[Theorem 2.9]{Kreulichevo} $v_n=\lim_{\la\to 0} v_{n,\la}$ is the integral solution in the sense of Definition \ref{def-int-sol-inhomo} with $f(t)=F(t,(u_n)_t)$ to
\begin{equation} 
\left\{
\begin{array}{rcl} 
v_{n+1}^{\prime}(t) &\in & B(t)v_{n+1}(t)+ \om v_{n+1}(t)+F(t,(u_n)_t): \ t \in \jz \\
v_{n+1}(0)&=&\vp(0).
\end{array}
\right.
\end{equation}
As $\lim_{\la\to 0} u_{n,\la}=u_n$ by Lemma \ref{induction_step_1} $v_n=u_n.$  Moreover, $u_{n+1}$ satisfies the inequality,
\begin{eqnarray*} 
\lefteqn{\nrm{u_{n+1}(t)-x}-\nrm{u_{n+1}(r)-x}} \\
&\le &\int_r^t\fk{[y+F(\nu,(u_n)_{\nu}),u_{n+1}(\nu)-x]_{+}+\om\nrm{u_{n+1}(\nu)-x}}d\nu \\
&&+\Lw(\nrm{x})\int_r^t\nrm{\hw(\nu)-\hw(r)}d\nu +\nrm{y}\int_r^t\nrm{g(\nu)-g(r)}d\nu
\end{eqnarray*}
for all $0\le r\le t \le T,$ and $s\in [a,b],$ $[x,y]\in B(s)+\om I.$ Passing to $n\to\infty$ completes the proof.
\end{proof}

For related results compare the studies of Ghavidel \cite{Ghavidel}, Ghavidel/Ruess \cite{GhavidelRuess},  Kartsatos \cite{bd-Kartsatos}, Kartsatos/Liu \cite{KartsatosLiu}, Kartsatos/Parrot \cite{KartsatosParrot}, Jeong/Shin \cite{Jeong}, Ruess/Summers \cite{Ruess_Summers_stability},\cite{Ruess_FDE_aut},  and Tanaka \cite{Tanaka}. Moreover, the equation is discussed in the textbook \cite{Ha}.

\section{Application}
Next we give a short application to asymptotically almost periodic functions, which extends \cite[Theorem 7.2]{Kreulichevo} to the cases of finite and infinite delay with a given initial history $\vp.$ Let
\begin{eqnarray*}
\lefteqn{AAP(\rep,X) } \\
&=&\bk{f\in BUC(\rep,X):\bk{t\mapsto f(t+s)}_{s\in\rep} \subset BUC(\rep,X)  \mbox{ is relatively compact} },\\
\lefteqn{AP(\re,X)} \\ 
&=&\bk{f\in BUC(\re,X):\bk{t\mapsto f(t+s)}_{s\in\re} \subset BUC(\re,X)  \mbox{ is relatively compact} }
\end{eqnarray*}
With the above definitions note that,
$$
AAP(\rep,X)=AP(\re,X)_{|\rep}\oplus C_0(\rep,X).
$$

Due to the fact that an almost periodic function $u$ is completely known if $u_{|\rep} $ is given we can define,
\begin{eqnarray*}
\lefteqn{AAP^+(\ji\cup\rep,X) } \\
&=&\bk{f\in BUC(\ji\cup\re,X):\bk{\rep\ni t\mapsto f(t+s)}_{s\in\rep} \subset BUC(\rep,X)  \mbox{ is relatively compact} },\\
\end{eqnarray*}
and we obtain with the projection 
$$
\Funk{P^+}{AAP^+(\ji\cup\rep,X)}{AP(\re,X)}{f}{P_a(f_{|\rep})}
$$
and 
$$C_0^+(\ji\cup\rep,X)=\bk{f\in BUC(\ji\cup \rep,X):\ilm{t}f(t)=0}$$
the decomposition
$$
AAP^+(\ji\cup\rep,X)=AP(\re,X)_{|\ji\cup\rep}\oplus C_0^+(\ji\cup\rep,X).
$$
We call $f^a:=P^+(f)$ the almost periodic part of $f\in AAP^+(\ji\cup\rep,X)$ as well.

If the resolvent of $B(t), $ $J_{\la}(t) $ satifies for all 
$x\in X,$ 
$\bk{t\mapsto J_{\la}(t)x} \in AAP^+(\ji\cup\rep,X)$ 
we conclude with $\bk{t\mapsto J_{\la}^a(t)x} \in AP(\re,X), $ and $\bk{t\mapsto J_{\la}^0(t)x}\in C_0^+(\ji\cup\rep,X)$
$$
\bk{t\mapsto J_{\la}(t)x}=\bk{t\mapsto J_{\la}^a(t)x} + \bk{t\mapsto J_{\la}^0(t)x}
$$
and the almost periodic parts $\bk{t\mapsto J_{\la}^a(t)x}$ are pseudo resolvents in the sense of \cite[Definition 7.1]{Miyadera}. Thus, we can define an almost-periodic part of $B(t),$ the operators 
$$
B^a(t)=\frac{1}{\la}\fk{I-(J_{\la}^a(t))^{-1}},
$$ 
compare \cite{KreulichSplit}.

In the forthcoming theorem it is shown that, even in the infinie delay case, an almost periodic solution is found. As consequence of the previous results it is found that not only the operator has to become almost periodic, the initial history as well.
\begin{theo}
Let $F:\rep\times E \to X$ satisfy the Assumption \ref{control-F(t,vp)},
$\vp\in E$ with $\vp(0)\in \overline{D},$
Let $B(t)$ satisfy Assumption \ref{general-IVA0} and either Assumption \ref{general-IVE1} with $K_0< -\om$ and $k$ uniformly continuous, or Assumption \ref{general-IVE2} with $\max(K_0,L_g)<-\om$ and $k$ Lipschitz.
Further, let
$$
Y=\bk{f \in BUC(\ji\cup\rep,X): f_{|\rep}\in AAP(\rep,X)}=AAP^+(\ji\cup\rep,X).
$$
If 
$$
\bk{t \mapsto J_{\la}^{\om}(t)x},\bk{t \mapsto F(t,\psi_t)} \in AAP^+(\ji\cup\rep,X) 
$$
for all $\psi\in Y,$ and $x\in X$ then the integral solution $u$ of (\ref{FDEonInterval}) is an element of $Y.$
Let for given $\vp\in E$ $\bk{ t\mapsto F^a (t,\vp) }$ the almost periodic part of $\bk{t\mapsto F(t,\vp)}.$ If
\begin{equation}\label{F(t_y_a)}
P^+(\bk{t\mapsto F(t,\psi_t)})=F^a(t,\psi^a_t),
\end{equation}
for all $\psi\in Y,$  then the almost periodic part $u^a$ of the solution $u$ is a generalized solution to the equation
$$
(u^a)^{\prime}(t)\in B^a(t)u^a(t)+\om u^a(t)+F^a(t,u^a_t), \mbox{ for all } t\in \rep.
$$

\end{theo}

\begin{proof}
We restrict the proof to Assumption \ref{general-IVE2}. As $A(t,\vp):=B(t)+F(t,\vp)$ satisfies Assumption \ref{IVFDE_1} and \ref{IVFDE_2} and $\max(K_0,L_g)<-\om$ in view of Theorem \ref{half-line-asymptotics} it remains to show 
$$
\bk{t \mapsto J_{\la}^{\om}(t,\psi_t)f(t)} \in AAP(\rep,X) \ \mbox{ for all } \psi\in Y, 
f\in AAP(\rep,X)
$$
Letting $J_{\la}^{\om}(t)$ the resolvent to $B(t)+\om I$, then
\begin{eqnarray*}
J_{\la}^{\om}(t,\psi_t)x&=&J_{\la}^{\om}(t)\fk{\la F(t,\psi_t)+x},
\end{eqnarray*}
and the proof concludes with 
applying the methods shown in the proof of \cite[Chapter VII,Lemma 4.1]{Dal-Krein}, which apply due to Assumption \ref{control-F(t,vp)}, and $\Jlw(t)$ contractive. To prove the second claim recall that mapping on $AAP(\rep,X),$
$$
F_{n+1,\la}(u):=\bk{t\mapsto \Jl^{\om}(t)\fk{\la F(t,(u_{\la,n})_t)+\nres{t}\vp(0)+\frac{1}{\la}\int_0^t\nres{\tau}u(t-\tau)d\tau}},
$$
is the approximating fix point equation, which lead to $\bk{u_{\la,n}}_{\la>0,n\in\za}$ uniformly convergent, more concrete $\lim_{n\to \infty}\lim_{\la\to 0}u_{\la,n}$ exists uniformly on $I\cup \rep.$ Consequently, their almost periodic parts are convergent on $\re$ as well. Moreover, the almost periodic parts become a fix point of the fix point mappings given by 
$$
u^a_{\la,n+1}=\bk{t\mapsto \Jl^{\om,a}(t)\fk{\la F^a(t,(u_{\la,n}^a)_t)+\frac{1}{\la}\int_0^{\infty}\nres{\tau}u^a_{\la,n+1}(t-\tau)d\tau}},
$$
compare the methods \cite{KreulichDiss} and \cite[Appendix]{KreulichSplit}.
Note that the above fix point equation is the one for the Yosida-approximation of the whole line equation, 
$$
\fk{\frac{\partial}{\partial t}}_{\la}\ul(t)\in B^a(t)\ul(t)+\om \ul(t)+F^a(t,(\ul)_t), \mbox{ for all } t\in \re.
$$
Compare \cite[Proof of Prop. 2.14, p. 1084]{Kreulichevo}, with the right hand side $f_{\la,n}(t)=F^a(t,(u_{\la,n}^a)_t).$
Due to the uniform convergence on $\ji\cup\rep$ we may pass to the limits $\la\to 0$ and $n\to \infty$ to obtain the desired generalized almost periodic solution.

\end{proof}

\begin{remk} In the case of finite delay the assumption (\ref{F(t_y_a)}) is given.
\end{remk}

\section{Appendix}
\begin{pro}\label{non-deceasing-convolution}
Let $f,g\in C[0,T]$ positive and $f$ non-decreasing, then the convolution
$$
\bk{t\mapsto \int_0^tg(\tau)f(t-\tau)d\tau}
$$
is non-decreasing.
\end{pro}
\begin{proof}
Let $t,s\in [0,T]$ with $t\ge s, $ then
\begin{eqnarray*}
\lefteqn{\int_0^tg(\tau)f(t-\tau)d\tau-\int_0^sg(\tau)f(s-\tau)d\tau } \\
&=& \int_0^sg(\tau)\fk{f(t-\tau)-f(s-\tau)}d\tau +\int_s^tg(\tau)f(t-\tau)d\tau.
\end{eqnarray*}
As the difference is positive we are done.
\end{proof}

\begin{lem} \label{s-t-integral-inequality} \label{ineq-finite-integral}
Let $t>a$, $\alpha>0$ and $u,f \in C[0,T]$ such that
$$
u(t)\le f(t)+\alpha \int_a^t\exp(-\beta(t-\tau))u(\tau)d\tau.
$$
Then
$$
u(t)\le f(t)+\alpha\int_a^t\exp((\alpha-\beta)(t-\tau))f(\tau)d\tau.
$$
\end{lem}

\begin{pro} \label{quasinilpotent-boundedness}
Let $X$ be a Banach lattice with $ \nrm{\btr{x}}\le\nrm{x},$ and $T,S:X\to X$ positive then:
\begin{enumerate}
\item If for all $x\in X, x\ge 0$ $Tx\le Sx,$ then $\nrm{T^n}\le\nrm{S^n}$ for all $n\in\za.$
\item If $0\le \la \le q$ then $\la I+S \le qI+S $ and consequently 
$\nrm{(\la I+S)^n}\le \nrm{(qI+S)^n}.$ 
\item Let $T$ additionally a quasi-nilpotent operator, $0<q<1$, 
Then, if for positive $g \in X$ and positive $\bk{f_n^{\la}}_{\la >0,n\in\za}\subset X,$ with $\bk{f_1^{\la}}_{0<\la\le q}$ bounded, and 
\begin{equation}\label{boundedness-recursion}
f_{n+1}^{\la}\le g+(\la I+T)f_{n}^{\la},
\end{equation}
we have,
$
\bk{f_n^{\la}}_{0<\la \le q ,n\in\za}
$
is bounded in $X.$
\end{enumerate}
\end{pro}

\begin{proof} Let $x\in B_X.$
The first claim is an induction. Let $ n=1,$   by the positivity of $T$  \cite[p. 58, -2]{SchaeferLattice}, and the prerequisite $T\le S$ we have

$$ \btr{Tx} \le  T\btr{x} \le S\btr{x}.$$ 
As X is Banach lattice \cite[Def. 5.1, p. 81]{SchaeferLattice} we have 
$\btr{x}\le\btr{y} $ implies $\nrm{x}\le \nrm{y},$ and obtain by the assumption $ \nrm{\btr{x}}\le\nrm{x},$ that $\nrm{Tx}\le \nrm{S\btr{x}} \le\nrm{S}.$ Thus we have $\nrm{T}\le\nrm{S}.$ \\
$n\to n+1:$ 

By the positivity of $T^{n+1}$, and the hypothesis 
$
T^n\btr{x} \le S^n\btr{x}.
$
we have
$$
\btr{T^{n+1}x} \le T^{n+1}\btr{x}=TT^n\btr{x} \le TS^n \btr{x} \le S^{n+1}\btr{x}.
$$
Again the Banach lattice property \cite[Def. 5.1, p. 81]{SchaeferLattice} leads to
$\nrm{T^{n+1}x}\le \nrm{S^{n+1}\btr{x}}\le \nrm{S^{n+1}}$
which finishes the induction.
The second claim is direct consequence of claim 1.
As $T$ in (3) is assumed to be quasinilpotent, the spectral radius $r(\la I+T)\le r(qI+T)\le q.$ 
The recursion (\ref{boundedness-recursion}) leads with an induction to,
\begin{eqnarray*}
f_{n+1}^{\la}&\le&\sum_{k=0}^{n-1}(\la I+T)^k g + (\la I +T)^n f^{\la}_1,\\
&\le&\sum_{k=0}^{n-1}\nrm{(\la I+T)^k} \nrm{g}+ \nrm{(\la I +T)^n} \nrm{f^{\la}_1},\\
&\le&\sum_{k=0}^{n-1}\nrm{(qI+T)^k} \nrm{g}+ \nrm{(qI +T)^n} \nrm{f^{\la}_1},\\
\end{eqnarray*}
which proves the claim 3.
\end{proof}
From \cite{Kreulichevo} we have the following Proposition.
\begin{pro} \label{bd-limsup-Lipschitz}
Let $f\in C((0,T),X),$ s.t.
$$ \limsup_{s \to 0} \frac{1}{\btr{s}}\nrm{f(t+s)-f(t)}\le L,\quad \forall t\in(0,T),$$
then $f$ is Lipschitz with a Lipschitz-constant less or equal $L.$
\end{pro}
As a direct consequence we have:

\begin{lem}  \label{limsup_sup_compare}
Let $f$ Lipschitz on $(a,b),$ then
\begin{eqnarray*}
\lefteqn{\inf {\bk{L>0 : \nrm{f(\alpha)-f(\beta)} \le L \btr{\alpha-\beta}, \ \forall \ a<\alpha,\beta<b }} } \\
&\le& \sup_{\alpha\in(a,b)} \limsup_{s\to 0} \frac{1}{\btr{s}}\nrm{f(\alpha-s)-f(\alpha)}
\end{eqnarray*}
and 
\begin{eqnarray*}
\sup_{\alpha\in(a,b)}\limsup_{s\to 0}\frac{1}{\btr{s}}\nrm{f(\alpha-s)-f(\alpha)} 
&\ge& \limsup_{s\to 0} \sup_{\alpha\in(a,b)}\frac{1}{\btr{s}}\nrm{f(\alpha-s)-f(\alpha)}.
\end{eqnarray*}
\end{lem}
\begin{proof}
The first claim is a direct consquence of Proposition \ref{bd-limsup-Lipschitz}. For the second claim note that,

\begin{eqnarray*}
\lefteqn{ \sup_{\alpha, \alpha -s \in (a,b)}\frac{1}{\btr{s}}\nrm{f(\alpha-s)-f(\alpha)} } \\
&\le&\inf\bk{L>0:\nrm{f(\alpha)-f(\beta)}\le L\btr{\alpha-\beta}, \ \forall \ a<\alpha,\beta<b}, \\
&\le&  \sup_{\alpha\in(a,b)}\limsup_{s\to 0}\frac{1}{\btr{s}}\nrm{f(\alpha-s)-f(\alpha)}.
\end{eqnarray*}
Now, we can pass to the $\limsup$ on the left hand side.
\end{proof}

\end{document}